\documentclass[reqno,11pt]{amsart}
\usepackage{amsmath,amssymb,latexsym,soul,cite,mathrsfs,accents}
\usepackage{color,enumitem,graphicx}
\usepackage[colorlinks=true,urlcolor=blue,
citecolor=red,linkcolor=blue,linktocpage,pdfpagelabels,
bookmarksnumbered,bookmarksopen]{hyperref}
\usepackage[english]{babel}
\usepackage[left=2.6cm,right=2.6cm,top=2.9cm,bottom=2.9cm]{geometry}
\usepackage{tensor}
\newtheorem{theorem}{Theorem}
\newtheorem{lemma}[theorem]{Lemma}
\newtheorem{corollary}[theorem]{Corollary}

\newtheorem{remark}[theorem]{Remark}
\newtheorem{definition}[theorem]{Definition}

\newtheorem{propositionletter}{Proposition}

\newtheorem{lemmaletter}{Lemma}

\newenvironment{acknowledgement}{\noindent\textbf{Acknowledgments.}}{}

\newtheorem*{innerthm}{\innerthmname}
\newcommand{\innerthmname}{}
\newenvironment{statement}[1]
{\renewcommand{\innerthmname}{#1}\innerthm}
{\endinnerthm}
\theoremstyle{definition}

\makeatletter
\def\namedlabel#1#2{\begingroup
	#2%
	\def\@currentlabel{#2}%
	\phantomsection\label{#1}\endgroup
}

\def\XXint#1#2#3{{\setbox0=\hbox{$#1{#2#3}{\int}$ }
		\vcenter{\hbox{$#2#3$ }}\kern-.6\wd0}}

\newcommand*\owedge{\mathpalette\@owedge\relax}
\newcommand*\@owedge[1]{%
	\mathbin{%
		\ooalign{%
			$#1\m@th\bigcirc$\cr
			\hidewidth$#1\m@th\wedge$\hidewidth\cr
		}%
	}%
}
\makeatother

\newcommand{\ud}{\mathrm{d}}
\newcommand{\loc}{\mathrm{loc}}

\newcommand{\Ss}{\mathbb{S}}

\DeclareMathOperator{\trace}{tr}

\title[Nonuniqueness results for constant sixth order $Q$-curvature metrics]{Nonuniqueness results for constant sixth order $Q$-curvature metrics on spheres with higher dimensional singularities}  
\thanks{This work was partially supported by Funda\c c\~ao de Amparo \`a Pesquisa do Estado de S\~ao Paulo (FAPESP), Conselho Nacional de Desenvolvimento Cient\'ifico e Tecnol\'ogico (CNPq), and Natural Sciences and Engineering Research Council of Canada (NSERC). 
	J.H.A. was supported by FAPESP \#2020/07566-3 and \#2021/15139-0. 
	P.P. was supported by FAPESP \#2016/23746-6 and CNPq \#313773/2021-1. 
	J.W. was supported by Natural Sciences and Engineering Research Council of Canada (NSERC)}

\author[J.H. Andrade]{Jo\~{a}o Henrique Andrade}
\author[P. Piccione]{Paolo Piccione}
\author[J. Wei]{Juncheng Wei}

\address[J.H. Andrade]{
	Department of Mathematics,
	University of British Columbia
	\newline\indent 
	V6T 1Z2, Vancouver, Canada
	\newline\indent
	and
	\newline\indent
	Institute of Mathematics and Statistics,
	University of S\~ao Paulo
	\newline\indent 
	05508-090, S\~ao Paulo, Brazil}
\email{\href{mailto:andradejh@math.ubc.ca}{andradejh@math.ubc.ca}}
\email{\href{mailto:andradejh@ime.usp.br}{andradejh@ime.usp.br}}

\address[P. Piccione]{Institute of Mathematics and Statistics,
	University of S\~ao Paulo
	\newline\indent 
	05508-090, S\~ao Paulo, Brazil}
\email{\href{mailto:piccione@ime.usp.br}{piccione@ime.usp.br}}

\address[J. Wei]{
	Department of Mathematics,
	University of British Columbia
	\newline\indent 
	V6T 1Z2, Vancouver-BC, Canada}
\email{\href{mailto:jcwei@math.ubc.ca}{jcwei@math.ubc.ca}}


\begin{document}
	
	\begin{abstract}
		We prove nonuniqueness results for constant sixth order $Q$-metrics on complete locally conformally flat $n$-dimensional Riemannian manifolds with $n\geqslant 7$.
		More precisely, assuming a positive Green function exists for the sixth order GJMS operator, our objective is two-fold.
		First, we use a classical bifurcation technique to prove that there exists infinitely many constant $Q$-curvature metrics on $\mathbb{S}^1\times\mathbb{S}^{n-1}$.
		As a by-product, we find the sixth order Yamabe invariant on this product manifold can be arbitrarily close to that of the round dimensional sphere, generalizing a result of Schoen about the classical Yamabe invariant.
		Second, when the underlying manifold is noncompact, we apply a bifurcation technique on Riemannian covering to construct infinitely many complete metrics with constant sixth order $Q$-curvature conformal to $\mathbb{S}^{n_1} \times \mathbb{R}^{n_2}$ or  $\mathbb{S}^{n_1} \times \mathbb{H}^{n_2}$, where $n_1+n_2\geqslant 7$. 
		Consequently, we obtain infinitely many solutions to the singular constant GJMS equation on round spheres $\mathbb{S}^n\setminus \mathbb{S}^k$ blowing up along a minimal equatorial subsphere with $0 \leqslant k<\frac{n-6}{2}$; this dimension restriction is sharp in the topological sense.
\end{abstract}

\maketitle

\begin{center}
	\footnotesize
	\tableofcontents
\end{center}

\numberwithin{equation}{section} 
\numberwithin{theorem}{section}

\section{Introduction}\label{section1}
In recent years, active research has been made on analogs of the Yamabe problem and its singular counterpart. 
In each of these problems, one seeks a representative of a conformal class with prescribed constant curvature of some type.
In the classical case, the scalar curvature functional is naturally the first one to be studied.
In more modern examples, fully nonlinear $\sigma_\ell$-curvatures with $\ell\in\mathbb N$ and fractional higher order $Q_{2\sigma}$-curvatures with $\sigma\in\mathbb R_+$ are also considered. 

In this manuscript, we focus on the integer case $\sigma=m\in\mathbb N$, more precisely $m=3$.
Let $(M^n, g)$ with $n \geqslant 7$ be $n$-dimensional Riemannian manifold; $g$ is called the background metric or (ambient metric), which will be fixed throughout the paper.    
In consonance with the plethora of results for the Yamabe problem, the scope of nonuniqueness in the constant sixth order $Q$-curvature problem ought to be much more prosperous, both geometrically and topologically. 
Our main purpose here is to exhibit an extensive nonuniqueness phenomena on a broad class of Riemannian manifolds.

We set $A_g,W_g,B_g$ to be respectively the {\it Schouten tensor}, {\it Weyl tensor}, {\it Bach tensor} given by
\begin{align}\label{geometrictensors}
	A_g&:=\frac{1}{n-2}\left(\operatorname{Ric}_{g}-\frac{1}{2(n-1)} R_{g} g\right)\\
	W_g&:=\accentset{\circ}{\rm Rm}_g-A_g\owedge g\\
	B_g&:=\Delta_{g} A_g-\nabla_g^{2} \operatorname{tr}_gA_g+2 \accentset{\circ}{\rm Rm}_g\cdot A_g-(n-4) A_g\times A_g-\|A_g\|^{2} g-2(\operatorname{tr}_g A_g) A_g
\end{align}	
where these expressions are written in an abstract index-free manner (see Section~\ref{sec:notation}).

We define the concept of higher order curvatures as follows
\begin{align}\label{curvatures}
	\nonumber
	Q_2(g)&:=\frac{n-2}{4(n-1)}R(g)\\
	Q_4(g)&:=-\Delta_g\sigma_1(A_g)+4\sigma_2(A_g)-\frac{n-4}{2}\sigma_1(A_g)^2&\\\nonumber
	Q_6(g)&:=-3 ! 2^6 v_6(g)-\frac{n+2}{2} \Delta_g(\sigma_1(A_g)^2)+4 \Delta_g\|A_g\|^2-8 \delta_g(A_g \#\ud \sigma_1(A_g))+\Delta_g^2 \sigma_1(A_g)\\\nonumber
	&-\frac{n-6}{2} \sigma_1(A_g) \Delta_g \sigma_1(A_g) -4(n-6) \sigma_1(A_g)\|A_g\|^2+\frac{(n-6)(n+6)}{4} \sigma_1(A_g)^3.
\end{align}
Associated with these curvatures, we have the following conformally invariant operators
\begin{align}\label{operators}
	\nonumber
	&P_2(g):=-\Delta_g+\frac{n-2}{2}R(g)&\\
	&P_4(g):=\Delta_{g}^{2}-\delta_g T_2(g) \#\ud+\frac{n-4}{2}Q_4(g)&\\\nonumber
	&P_6(g):=-\Delta_g^3-\Delta_g \delta_g T_2(g) \#\ud-\delta_g T_2(g) \#\ud \Delta_g-\frac{n-2}{2} \Delta_g\left(\sigma_1(A_g) \Delta_g\right)-\delta_g T_4(g) \#\ud+\frac{n-6}{2} Q_6(g),&
\end{align}
where 
\begin{align*}
	T_2(g)&:=(n-2)\sigma_1(A_g) g-8A_g,\\
	T_4(g)&:=-\frac{3 n^2-12 n-4}{4} \sigma_1(A_g)^2 g+4(n-4)\|A_g\|^2 g+8(n-2) \sigma_1(A_g) A_g\\
	&+(n-6)\Delta_g \sigma_1(A_g) g+48A_g^2-\frac{16}{n-4} B_g,\\
	v_6(g)&:=-\frac{1}{8}\sigma_3(A_g)-\frac{1}{24(n-4)}\langle B_g, A_g\rangle_g,
\end{align*}
with $\sigma_\ell$ is the $\ell$-th elementary symmetric function for each $\ell\in\mathbb N$.  
Notice that $P_2(g)=L(g)$ is the so-called conformal Laplacian, and $P_6(g),P_6(g)$ are their conformally invariant powers.   

This class of higher order operators on manifolds was first introduced in the works of Graham, Jenne, Manson, and Sparling \cite{MR1190438}, which is construction is based on earlier constructions of conformally invariant fourth order operators by the works of Paneitz \cite{MR2393291} and Branson \cite{MR904819}, that is, for $m=2$.
Given 
$m \in \mathbb{N}$ and a compact $n$-dimensional Riemannian manifold 
$(M^n,g)$ with $n>2m$, the so-called GJMS operators are conformally 
covariant differential operators, denoted by $P_{2m}(g)$, satisfying that
the leading order term of $P_{2m}(g)$ is 
$(-\Delta_{g})^m$. 	
The construction of these operators are based on the seminal work of Fefferman and Graham \cite{MR2858236} on ambient metrics for any $m\in\mathbb N$.
From this point of view, one can then construct the associated $Q$-curvature 
of order $2m$ by $Q_{2m}(g) = P_{2m}(g) (1)$. 
When $m=1$,
one recovers the conformal Laplacian 
\begin{equation*}
	P_2(g) = -\Delta_g + \frac{n-2}{4(n-1)} R_g \quad {\rm with} \quad Q_2(g) =\frac{n-2}{4(n-1)} R_g,
\end{equation*}
where $\Delta_g$ is the Laplace--Beltrami operator of $g$ and $R_g$ is 
its scalar curvature. 
For the standard round sphere, Grahan and Zworski \cite{MR1965361} and Chang and 
Gonz\'alez \cite{MR2737789} 
extended these definitions to obtain (nonlocal) operators $P_{2\sigma}(g)$ of any 
order $\sigma\in \mathbb R_+$ with $n>2\sigma$ as well as the corresponding $\sigma$-fractional order $Q$-curvatures. 
Once again, the leading order part of $P_{2\sigma}(g)$ is $(-\Delta_{g})^{\sigma}$, 
understood as the principal value of a singular integral operator. 
Although we have explicit
formulas for $P_2(g)$, $P_4(g)$ and $P_6(g)$,
the expressions for $P_{2\sigma}(g)$ and $Q_{2\sigma}(g)$ for a general $\sigma\in\mathbb R_+$ are 
far more complicated (see, for instance \cite{MR3694655}).
For more details, we refer the interested reader to \cite{MR3077914,MR3652455,MR4285731,MR3073887}.

In general, the higher order $Q_{2m}$-curvatures defined in \eqref{curvatures} for $m=1,2,3$ transform nicely under a conformal change, namely for any conformal metric $\bar{g}\in[g]$, one has 
\begin{equation}\label{transformationlawcurvature}
	Q_{2m}({\bar{g}})=\frac{2}{n-2m} u^{-\frac{n+2m}{n-2m}} P_{2m}(g)u,
\end{equation}
where $P_{2m}(g):\mathcal{C}^{\infty}(M)\rightarrow\mathcal{C}^{\infty}(M)$ is the {\it GJMS operator} defined in \eqref{operators} and 
\begin{equation*}
	[g]:=\left\{\bar{g}\in{\rm Met}^\infty(M): \bar{g}=u^{{4}/{(n-2m)}} g \ {\rm for \ some} \ u \in \mathcal{C}_+^{\infty}(M)\right\}
\end{equation*}
is the conformal class of $g$,  where $u\in\mathcal{C}_+^{\infty}(M)$ if and only if $u\in\mathcal{C}^{\infty}(M)$ and $u>0$.
This operator is known to satisfy the transformation law below
\begin{equation}\label{transformationlawoperator}
	P_{2m}(g)(\phi)=u^{-\frac{n+2m}{n-2m}} P_{2m}(g)(u\phi).
\end{equation}
It is straightforward to see that one recovers \eqref{transformationlawcurvature} by substituting $\phi\equiv 1$ into \eqref{transformationlawoperator}. 
The transformation rules \eqref{transformationlawcurvature} and \eqref{transformationlawoperator} motivate us to define the so-called {\it total curvature functional} $\mathcal{Q}_{2m}:[g] \rightarrow \mathbb{R}$ and {\it Yamabe-type functional} $\mathcal{Y}^+_{2m}:[g] \rightarrow \mathbb{R}$ for $m=1,2,3$ as
\begin{equation*}
	\mathcal{Q}_{2m}(\bar{g}):=\frac{\int_M Q_{2m}({\bar{g}}) \ud \mathrm{v}_{\bar{g}}}{{\rm vol}_{\bar{g}}(M)^{\frac{n-2m}{n}}}= \frac{\frac{2}{n-2m}\int_M u P_{2m}(g)u \ud \mathrm{v}_g}{\left(\int_M u^{\frac{2 n}{n-2m}} \ud \mathrm{v}_g\right)^{\frac{n-2m}{n}}}
\end{equation*}
and
\begin{align*}
	\mathcal{Y}^+_{2m}(M, g) :=\inf_{\bar{g} \in[g]} \mathcal{Q}_{2m}(\bar{g})=\inf_{\bar{g} \in[g]} \frac{\int_M Q_{2m}({\bar{g}}) \ud \mathrm{v}_{\bar{g}}}{{\rm vol}_{\bar{g}}(M)^{\frac{n-2m}{n}}}=\inf_{u \in \mathcal{C}_+^{\infty}(M)} \frac{\frac{2}{n-2m} \int_M u P_{2m}(g)u \ud \mathrm{v}_g}{\left(\int_M u^{\frac{2 n}{n-2m}} \ud \mathrm{v}_g\right)^{\frac{n-2m}{n}}},
\end{align*}
where $\ud \mathrm{v}_g$ denotes the standard volume measure.
This quantity is a higher order analog of the famous {\it Yamabe invariant}.  
Following this notation, for $m=1,2,3$ we set the differential invariant
\begin{equation*}
	\mathbb{Y}_{2m}^{+}(M):=\sup_{[g] \in \mathfrak{c}} \mathcal{Y}_{2m}^{+}(M, g)=\sup _{[g] \in \mathfrak{c}} \inf_{\widetilde{g} \in[g]}\left\{\frac{\int_M Q_{2m}({\bar{g}}) \ud \mathrm{v}_{\bar{g}}}{{\rm vol}_{\bar{g}}(M)^{\frac{n-2m}{n}}}\right\},
\end{equation*}
where $\mathfrak{c}$ denotes the space of conformal classes on the underlying manifold. 

Notice that we require all test functions in the infimum for $\mathcal{Y}_{2m}^{+}$must all be positive. 
Naturally, one may also define
\begin{equation*}
	\mathcal{Y}_{2m}(M, g):=\inf_{\mathcal{C}^{\infty}(M)\setminus\{0\}}\frac{\frac{2}{n-2m} \int_M u P_{2m}(g)u \ud\mathrm{v}_g}{\left(\int_M|u|^{\frac{2 n}{n-2m}} \ud\mathrm{v}_g\right)^{\frac{n-2m}{n}}},
\end{equation*}
and clearly $\mathcal{Y}_{2m}(M, g) \leqslant \mathcal{Y}_{2m}^{+}(M, g)$.

We assume that $(M^n,g)$ is closed (compact and without boundary). 
We seek conformal metrics of the form $\bar{g} = \mathrm{u}^{{4}/{(n-6)}} g$. 
We prescribe the resulting metric to have constant $Q_6$-curvature, which we normalize to 
be 
\begin{equation}\label{const_q_curv}
	Q_n = Q_6(g_0) = \frac{n(n^4-20n^2+64)}{2^5},
\end{equation} 
where $g_0=g_{\mathbb S^n}$ denotes the standard round metric on the sphere.
In addition, using the transformation law \eqref{transformationlawoperator}, the condition that $\bar{g} = \mathrm{u}^{{4}/{(n-6)}} g$ satisfies $Q_6(g) = Q_n$ on 
$M$ is equivalent to the finding $\mathrm{u}\in \mathcal{C}^\infty(M)$ solution to the the sixth order GJMS equation 
\begin{equation}\tag{$\mathcal{Q}_{6,g}$}\label{ourequationmanifold}
	P_6({g})\mathrm{u}=c_n \mathrm{u}^{\frac{n+6}{n-6}} \quad \mbox{on} \quad M,
\end{equation}
where $c_n=\frac{n-6}{2}{Q}_n$ is a normalizing constant.    
The operator on the left-hand side is the sixth order GJMS operator on the sphere defined in \eqref{operators}. 
We remark that the nonlinearity 
on the right-hand side of \eqref{ourequationmanifold} has critical growth with respect 
to the compact embedding of the Beppo--Levi space $\mathcal{D}^{3,2}(\mathbb R^n)\rightarrow L^{2^{\#}}(\mathbb R^n)$, where $2^\# = \frac{2n}{n-6}$. It is well known that this 
embedding is not compact, reflecting the conformal invariance of the PDE \eqref{ourequationmanifold}. 

As usual, we need to analyze the conformal invariant to find solutions to the last equation.
Namely, we define the sixth order Yamabe invariant $\mathcal{Y}_{6}:[g]\rightarrow\mathbb R$ as
\begin{equation}\label{yamabeinvariant}
	\mathcal{Y}_{6}(M,g):=\inf_{u \in L^{\frac{2n}{n-6}}(M) \setminus\{0\}} \frac{\frac{2}{n-6} \int_M u P_6(g)u \ud\mathrm{v}_g}{\left(\int_M|u|^{\frac{2 n}{n-6}} \ud\mathrm{v}_g\right)^{\frac{n-6}{n}}},
\end{equation}
where $\bar{g}=u^{{4}/{(n-6)}} g$ is a conformal metric with $u\in L^{\frac{2n}{n-6}}(M)$ is a conformal factor.	
The main results in this paper are based on the existence of a positive Green function for the sixth order GJMS operator, which can be stated as 
\begin{enumerate}
	\item[\namedlabel{itm:P}{({\rm $\mathcal{G}_+$})}] For each $p\in M$, there exists a positive Green function, denoted by $G_6(g)_p:M\times M\rightarrow \mathbb R$, for the sixth order GJMS operator $P_6(g):\mathcal{C}^\infty(M)\rightarrow \mathcal{C}^\infty(M)$ with a pole at $p$. In other terms, it satisfies
	\begin{equation*}
		P_6(g)G_6(g)_p=\delta_p \quad {\rm in} \quad M,
	\end{equation*}
	where $\delta_p$ denotes the Dirac measure centered at $p$.
\end{enumerate}   
In the light of the last definition, we can use the positive Green function for the sixth order GJMS operator to define the dual conformal invariant.
This quantity is a counterpart to the new invariant introduced by Hang and Yang \cite{MR3518237} for the fourth order setting.
As a matter of fact, assuming that $(M^n,g)$ satisfies \ref{itm:P}, we define the dual sixth order Yamabe invariant $\Theta_6:[g]\rightarrow\mathbb R$ as
\begin{equation}\label{dualyamabeinvariant}
	\Theta_6(M, g):=\sup_{f \in L^{\frac{2n}{n+6}}(M) \setminus\{0\}} \frac{\int_M f G_6(g) f \ud{\rm v}_{g}}{{\left(\int_M f^{\frac{2 n}{n+6}} \ud \mathrm{v}_g\right)^{\frac{n-6}{n}}}}=\sup_{\bar{g} \in[g]} \frac{\int_M Q_6({\bar{g}}) \ud\mathrm{v}_{\bar{g}}}
	{\|Q_6({\bar{g}})\|^2_{L^{\frac{2 n}{n+6}}(M)}},
\end{equation}
where $\bar{g}=f^{{4}/{(n+6)}} g$ is a conformal metric with $f\in L^{\frac{2n}{n+6}}(M)$ is a dual conformal factor.      
In some sense, $\Theta_6(M, g)$ is comparable with the reciprocal $\mathcal{Y}_6^{+}(M, g)^{-1}$. 
The advantages of considering this quantity are that if a maximizer for \eqref{dualyamabeinvariant} exists, it is smooth and does not change its sign, and the associated conformal metric has constant sixth order $Q$-curvature. 

One needs {\it a priori} bound on the sixth order Yamabe invariant to apply the topological bifurcation strategy for coverings. 
This type of inequality is related to the asymptotic of the Green function and the concept of mass for a GJMS operator \cite{MR888880,arXiv:1012.4414}.
Below we state the positive mass theorem of Qing and Raske \cite{MR2219215} (see also \cite{arXiv:1012.4414}).
For this, we work in conformal normal coordinates \cite{MR991959,MR1225437} around $p$, {\it i.e.} $r:=\ud_g(x,p)$.	

\begin{lemmaletter}\label{prop:greenfunction1}
	Let $(\widetilde{M}^n,\tilde{g})$ be a closed locally conformally flat manifold with $n\geqslant 7$ satisfying $R(\tilde{g})\geqslant0$. 
	Suppose that \ref{itm:P} holds.   
	Then, for any $p\in M$, one has
	\begin{equation*}
		G_6(g)_p(r)=C_n r^{6-n}+m_6(g)_p+\mathcal{O}(r) \quad {\rm as} \quad r\rightarrow0,
	\end{equation*}
	where $m_6(g)\in\mathbb R$ is a constant.
	Moreover, if the Poincar\'e exponent of the Kleinian group $\Gamma\subset {\rm iso}(M)$ satisfies
	\begin{equation}\label{qing-raske}\tag{$\mathcal{QR}$}
		0<\delta(\Gamma)<\frac{n-6}{2},
	\end{equation}
	where $\delta(\Gamma):=\inf \{\delta>0: \sum_{\gamma \in \Gamma}\left|\gamma^{\prime}(x)\right|^\delta<\infty \ {\rm for \ all} \ x\in\mathbb{S}^n\}$,
	then $m_6(g)\geqslant0$,  with equality $m_6(g)=0$ implying $(\widetilde{M}^n,\tilde{g})\simeq(\mathbb{R}^n,\delta)$.
	
\end{lemmaletter}

Based on the positivity and topological properties above, we have the Aubin-type result below
\begin{propositionletter}\label{lm:aubin}
	Let $(\widetilde{M}^n,\tilde{g})$ be a closed locally conformally flat manifold with $n\geqslant 7$ satisfying $R(\tilde{g})\geqslant0$.
	Suppose that \ref{itm:P} and \eqref{qing-raske} hold.   
	Then, both the statements below hold and are equivalent
	\begin{enumerate}
		\item[\namedlabel{itm:A}{({\rm $\mathcal{A}$})}] The infimum in \eqref{yamabeinvariant} is attained at some smooth function $u \in \mathcal{C}^{\infty}(M)$, and the conformal metric $\bar{g}=u^{{4}/{(n-6)}} g$ has constant sixth order $Q$-curvature. Moreover, one has
		\begin{equation*}
			\mathcal{Y}_6(M, g) \leqslant \mathcal{Y}_6(\mathbb{S}^n, g_0),
		\end{equation*}
		with equality if and only if $(M^n, g)$ is conformally equivalent to the round sphere $(\mathbb{S}^n, g_0)$;
		\item[\namedlabel{itm:A'}{(\rm $\mathcal{A}^\prime$)}]
		The supremum in \eqref{dualyamabeinvariant} is attained at some smooth function $f \in \mathcal{C}^{\infty}(M)$, and the conformal metric $\bar{g}=f^{{4}/{(n+6)}} g$ has constant sixth order $Q$-curvature. Moreover, one has
		\begin{equation*}
			\Theta_6(M, g) \geqslant \Theta_6(\mathbb{S}^n, g_0),
		\end{equation*}
		with equality if and only if $(M^n, g)$ is conformally equivalent to the round sphere $(\mathbb{S}^n, g_0)$.
	\end{enumerate}     
\end{propositionletter}

We recall that on Einstein manifolds, the sixth order GJMS operator expresses as \eqref{factorization}.
In this case, one can directly apply Bishop--Gromov volume comparison to guarantee that the geometric inequalities above are satisfied.
This follows if one assumes the validity of the positive mass theorem (which can be recovered by assuming \eqref{qing-raske}) in the low-dimensional cases $n=7,8,9$ (c.f. \cite{MR2219215,arXiv:1012.4414}) or from the local structure around a non-vanishing point of the Weyl tensor in remaining cases $n\geqslant 10$ (c.f. \cite{MR3652455}).
Recently, there has been some progress in proving the validity of \ref{itm:P}.
Indeed, Jeffrey Case informed us of a sufficient condition for \ref{itm:A} to be satisfied for a special Einstein product introduced by Gover and Leitner \cite{MR2574315}.  
Namely, when the underlying manifold satisfies the property below
\begin{enumerate}
	\item[\namedlabel{itm:GL}{({\rm $\mathcal{GL}$})}] $(M^n,g)=(M^{n_1}\times M^{n_2},g_1\oplus g_2)$ with $(M^{n_1},g_1)$ and $(M^{n_2},g_2)$ being both Einstein and 
	\begin{equation*}
		n_1(n_1-1)R(g_1)+n_2(n_2-1)R(g_2)=0.
	\end{equation*}
\end{enumerate} 
with $R(g_2)>0$ and $n_2\gg1$ sufficiently large, then \ref{itm:P} is true.  
In addition, he also expects that \ref{itm:A'} holds in this case.

Our objective is to prove nonuniqueness results for constant sixth order $Q$-curvature metrics.
Although these results look, at first glance, seemingly different, they are both based on bifurcation techniques and global geometric inequalities above, namely \ref{itm:A} and \ref{itm:A'}.

The first main result of this paper is based on \cite[Theorem 1]{arXiv:2002.05939} and uses a bifurcation technique to obtain the existence of infinitely many on the product of two spheres.

\begin{theorem}\label{thm1}
	Let $(\mathbb S^1\times\mathbb S^{n-1},g_{\mathbb S^1}\oplus g_{\mathbb S^{n-1}})$  with $n\geqslant 7$.
	Then, there exists infinitely many nonhomothetic conformal metrics on ${\rm Met}^\infty(\mathbb S^1\times\mathbb S^{n-1})$ with constant sixth order $Q$-curvature.
	Moreover, one has 
	\begin{equation*}
		\lim_{\ell\rightarrow+\infty}\mathcal{Y}_6^{+}(\mathbb{S}^1 \times\mathbb{S}^{n-1},g_\ell)=\mathcal{Y}_6^{+}(\mathbb{S}^n,g_0).
	\end{equation*}
	Consequently, it follows the equality
	\begin{equation*}
		\mathbb{Y}_6^{+}(\mathbb{S}^1 \times\mathbb{S}^{n-1})=\mathbb{Y}_6^{+}(\mathbb{S}^n)=\frac{n(n^4-20n^2+64)}{2^5}\omega_n^{6/n}
	\end{equation*}
	cannot be realized by a smooth metric on $\mathbb{S}^{n-1} \times \mathbb{S}^1$.
\end{theorem}

Notice that $\mathbb{Y}_6^{+}(\mathbb{S}^n)$ coincides with the best Sobolev constant of the compact embedding of the Beppo--Levi space $\mathcal{D}^{3,2}(\mathbb R^n)\rightarrow L^{2^{\#}}(\mathbb R^n)$.
As a by-product of Lemma~\ref{lm:aubin} and Theorem~\ref{thm1}, we prove that the sixth order Yamabe invariant on $\mathbb{S}^1\times\mathbb{S}^{n-1}$ coincides with the one of the round sphere.
To the best of our knowledge, this provides the first class of examples for which the sixth order Yamabe invariant is directly computed; such a fact is remarkable and of independent interest.

\begin{corollary}\label{cor2}
	Let $(\mathbb S^1\times\mathbb S^{n-1},g_{\mathbb S^1}\oplus g_{\mathbb S^{n-1}})$ with $n\geqslant 7$.
	Suppose that \ref{itm:P} holds.
	Then, it follows
	\begin{equation*}
		\mathcal{Y}_6(M, g)= \mathcal{Y}^*_6(M, g) \leqslant \mathcal{Y}_6(\mathbb{S}^n,g_0),
	\end{equation*}
	where
	\begin{equation*}
		\mathcal{Y}_6^*(M, g)=\inf _{\bar{g} \in[g], Q_2(\bar{g})>0, Q_4(\bar{g})>0} \mathcal{Q}_6(\bar{g})
	\end{equation*}
	with the equality implying $(M^n,g)=(\mathbb S^n,g_0)$.
	Moreover, one has  
	\begin{equation*}
		\mathbb{Y}_6(\mathbb{S}^{n-1} \times\mathbb{S}^1)=\mathbb{Y}_6^{+}(\mathbb{S}^{n-1} \times\mathbb{S}^1)=\mathcal{Y}_6^*(\mathbb{S}^n,g_0).
	\end{equation*}      
\end{corollary}

The constant $Q$-curvature problem can also be posed on noncompact manifolds (possibly with a simple structure at infinity). 
The natural boundary condition is the completeness of the metric. For instance, this is trivially satisfied by metrics that descend to a compact quotient $M\setminus\Gamma$; we call these periodic solutions with period $\Gamma\subset{\rm iso}(M)$;  this is inspired by the classical work of Schoen and Yau \cite{MR931204}.
More precisely, we seek for complete metrics on $M \setminus\Lambda$ of the form $\bar{g} = \mathrm{u}^{{4}/{(n-6)}} g$, where $\Lambda^k\subset M^n$ is a $k$-dimensional submanifold.	
Namely, when $(M^n,g)$ is a compact locally conformally flat manifold with positive scalar curvature $R_g\geqslant0$, one has that $\Lambda:=\mathbb S^n\setminus\mathscr{D}(\widetilde{M})$ can be seen as the complement of image of the developing map $\mathscr{D}:\widetilde{M}\hookrightarrow\mathbb S^n$, where $\widetilde{M}$ is the universal cover of $M$.

To $g$ to be complete on $M^n\setminus 
\Lambda^k$, one has to impose $\liminf_{\ud_g(p,\Lambda)} \mathrm{u}(p) = +\infty$. 	
As before, writing $\bar{g} = \mathrm{u}^{{4}/{(n-6)}} g$, one has that $Q_6(g) = Q_n$ on 
$M \setminus 
\Lambda$ is equivalent to finding $\mathrm{u}\in \mathcal{C}^\infty(M \setminus \Lambda)$ solving
\begin{equation}\tag{$\mathcal{Q}_{6,g,\Lambda}$}\label{ourequationmanifoldsingular}
	P_6({g})\mathrm{u}=c_n \mathrm{u}^{\frac{n+6}{n-6}} \quad \mbox{on} \quad M \setminus \Lambda.
\end{equation}
We denote by $\Omega:=M\setminus \Lambda$ and ${\rm Met}^{\infty}(\Omega)=\cup_{j=1}^{\infty}{\rm Met}^{j}(\Omega)$
the set of smooth metrics over $\Omega$, 
where
${\rm Met}^{j}(\Omega)$ denotes the set of $\mathcal{C}^j$ metrics over $\Omega$ with $j\in\mathbb N$ is equipped with the Gromov-Hausdorff topology.	
In this direction, it is convenient to introduce some terminology.
For any $g\in{\rm Met}^\infty(\Omega)$, let us define 
the moduli space of complete metrics below 
\begin{align*}
	\mathcal{M}_{6,\Lambda}({g})=\left\{\bar{g}\in [g] : \ \text{$\bar{g}$ is complete on $\Omega$} \ {\rm and} \  Q_6(g)\equiv Q_n   \right\}\subset {\rm Met}^{\infty}(\Omega).	
\end{align*}
This set is called the moduli space of constant sixth order $Q$-curvature metrics singular at $\Lambda$.
When $\Lambda=\varnothing$, we shall denote $	\mathcal{M}_{6,\Lambda}({g})=\mathcal{M}_6({g})$.
Also when $g=g_0$, we simply denote $\mathcal{M}_{6,\Lambda}({g}_0)=\mathcal{M}_{6,\Lambda}$.

Our second main result is inspired by \cite[Theorem B]{MR4251294} and exploits the abundance of discrete cocompact groups on symmetric spaces to find infinitely many periodic solutions with different periods on noncompact manifolds, reflecting a genuinely noncompact phenomenon.
\begin{theorem}\label{thm2}
	Let $(M^n, g)$ be a closed Riemannian manifold with $n \geqslant 7$.
	Assume that \ref{itm:P} and \ref{itm:A'} hold.
	Suppose that $(M^n_{\infty}, g_{\infty}) \rightarrow(M^n, g)$ is a Riemannian covering whose group of deck transformations has infinite profinite completion. 
	Then, there exists infinitely many nonhomothetic conformal periodic metrics on $\mathcal{M}_6(g_C\oplus g_{\mathcal{S}})$ with constant positive sixth order $Q$-curvature.
\end{theorem}

As a by-product of the proof, we have the following result
\begin{corollary}\label{cor3}
	Let $(C^{n_1}\times \mathcal{S}^{n_2},g_C\oplus g_{\mathcal{S}})$ be a complete noncompact Riemannian manifold with $n_1+n_2=n\geqslant 7$.
	Assume that \ref{itm:P} and \ref{itm:A'} hold.
	Assume that
	$(C^{n_1}, g_C)$ is a closed manifold with constant scalar curvature, and $(\mathcal{S}^{n_2}, g_{\mathcal{S}})$ is a simply-connected symmetric space of noncompact or Euclidean type.        
	Then, there exists infinitely many nonhomothetic conformal periodic metrics on $\mathcal{M}_6(g_C\oplus g_{\mathcal{S}})$ with constant positive sixth order $Q$-curvature.
\end{corollary}

An immediate consequence of Corollary~\ref{cor3} is that one can find infinitely many complete metrics with constant $Q$-curvature on the round sphere blowing up along an equatorial subsphere.

\begin{corollary}\label{cor4}
	Let         
	$(\mathbb S^n\setminus\mathbb S^k,g_{0})$ be the standard round sphere singular along a minimal equatorial $k$-subsphere with $n\geqslant 7$ and $0\leqslant k<\frac{n-6}{2}$.
	Suppose that \ref{itm:P} holds.
	Then, there exist infinitely many complete conformal metrics on $\mathcal{M}_{6,\mathbb S^k}$ with constant positive sixth order $Q$-curvature.
\end{corollary}

If one assumes that the results by Case and Malchiodi \cite{case-malchiodi} holds, then assumption \ref{itm:P} in Corollary~\ref{cor2} could be changed by the dimension condition $n\geqslant 10$.
Furthermore, it is not hard to check that the manifolds satisfying \ref{itm:GL} would provide a large class of manifolds to apply this general bifurcation technique on Riemannian coverings.

The restriction on the dimension of the singular set $0\leqslant k<{(n-6)/2}$ is sharp in the sense that the Poincar\'e exponent of the holonomy representation $\Gamma\subset {\rm iso}({M})\subset \mathbb{H}^{k+1}$ is not well-defined otherwise \cite{MR2232210} for its fourth order counterpart.
We also remark that the positivity condition \ref{itm:P} is usually hard to verify, and it is taken as a hypothesis in some recent papers (c.f. \cite{arXiv:2007.10180,MR2558328}).

In some exceptional cases, this property can be proved.
For instance, in the locally conformally flat case, this assumption is intimately related with the size of the Poincar\'e exponent, which by the injectivity of the developing map, coincides with the Hausdorff dimension of the singular set, that is, $0<\delta(\Gamma)<k$.
Furthermore, observe that for $k=0$, the singular set is a finite set of points. In this case, in \cite{arXiv:2302.05770} it is proved compactness for this moduli space of metrics under certain conditions on the asymptotic necksize of the isolated singularities. 

The proof of our main theorems is mainly based on variational bifurcation theory and topological methods \cite{MR1037143,MR2859263,MR4167290}. 
The primary sources of difficulties rely on the need for maximum principle enjoyed by this type of operator.
We are restricted to manifolds of dimension $n \geqslant 7$, as the cases $3\leqslant n\leqslant5$ and especially $n=6$ require a separate discussion.
Recently, for the case of constant fractional $Q$-curvature metrics $\sigma\in(0,1)$, Bettiol, Gonz\'alez, and Maalaoui \cite{arXiv:2302.11073} also used bifurcation techniques to prove similar results on round spheres.
In a companion paper \cite{andrade-piccione-wei}, we intend to study constant sixth order $Q$-curvature metrics on closed manifolds. 

Let us compare our results with the existing literature.
In his pioneering work, Yamabe \cite{MR0125546} was the first to introduce these invariants and study the problem of finding a constant scalar curvature metric in a given conformal class. 	
After that, Aubin \cite{MR0431287} proved that $ \mathcal{Y}_2(M, g) \leqslant\mathcal{Y}_2(\mathbb{S}^n,g_0)$.
He also showed that whenever $\mathcal{Y}_2(M, g)<\mathcal{Y}_2(\mathbb S^n,g_0)$ is satisfied, one can find a smooth metric with constant scalar curvature metric $\bar{g} \in[g]$ such that $\mathcal{Q}_2(\bar{g})=\mathcal{Y}_2(M, g)$. 
At last, Schoen \cite{MR788292} proved that $\mathcal{Y}_2(M, g)<\mathcal{Y}_2(\mathbb{S}^n,g_0)$ for each conformal class $[g] \neq[g_0]$. 
In particular, $\mathcal{Y}_2(M, g)<\mathcal{Y}_2(\mathbb{S}^n,g_0)$ whenever $M$ is not equivalent to the round sphere.  

Nevertheless, the equality $\mathbb{Y}_2(M)=\mathbb{Y}_2(\mathbb{S}^n)$ may occur even when $M$ is not the sphere. 
In particular, he found in \cite{MR1173050} an explicit sequence of metrics $\{g_\ell\}_{\ell\in\mathbb N}\subset{\rm Met}^\infty(\mathbb{S}^1 \times \mathbb{S}^{n-1})$ such that 
\begin{equation*}
	\mathcal{Y}_2(\mathbb{S}^1 \times \mathbb{S}^{n-1},g_\ell) \rightarrow \mathcal{Y}_2( \mathbb{S}^n,g_0) \quad {\rm as} \quad  \ell \rightarrow \infty,
\end{equation*}
and so $\mathbb{Y}_2(\mathbb{S}^1 \times \mathbb{S}^{n-1})=\mathbb{Y}_2(\mathbb{S}^n)$. 
As the underlying manifolds are not diffeomorphic, the equality above cannot be realized by a smooth metric on $\mathbb{S}^1 \times \mathbb{S}^{n-1}$.
As a consequence of the maximum principle, the Yamabe invariant $\mathcal{Y}_2(M, g)$ is always nonnegative.
Therefore, the minimizers of the functional $\mathcal{Q}_2$ over all nontrivial functions in $H^{1}(M)$ will be automatically positive.   

Nonetheless, minimizers of $\mathcal{Y}_{2m}(M, g)$ with $m=2,3$, if exist, might change sign.  
On the fourth order setting, 
Esposito and Robert \cite{MR1942129} showed that $\mathcal{Y}_4^{+}(M, g)<+\infty$ for each conformal class $[g]$. 
Gursky, Hang and Lin \cite{MR3509928} proved that if $n\geqslant 6$ and if $\mathcal{Y}_2(M, g)>0$ and $\mathcal{Y}_4^{+}(M, g)>0$ then
\begin{equation*}
	\mathcal{Y}_4(M, g)=\mathcal{Y}_4^{+}(M, g)=\mathcal{Y}_4^*(M, g):=\inf _{\bar{g} \in[g], Q_2(\bar{g})>0} \mathcal{Q}_4(\bar{g}).
\end{equation*}    
Thereafter, Hang and Yang \cite{MR3518237} proved that if $\mathcal{Y}_2(M, g)>0$ and $Q_4(g) \geqslant 0$ with $Q_4(g) \not \equiv 0$, then
\begin{equation}\label{aubinineqfourth}
	\mathcal{Y}_4(M, g)=\mathcal{Y}_4^*(M, g) \leqslant \mathcal{Y}_4(\mathbb{S}^n,g_0)
\end{equation}
and that equality implies $(M^n, g)\simeq(\mathbb{S}^n, g_0)$.      
Also, under these hypotheses, there exists a smooth, constant $Q$-curvature metric $\bar{g} \in[g]$ such that $\mathcal{Q}_4(\bar{g})=\mathcal{Y}_4^*(M, g)$.
The result referenced above implies the equality $\mathbb{Y}_4^{+}(\mathbb{S}^{n-1} \times\mathbb{S}^1)=\mathcal{Y}_4^*(\mathbb{S}^n)$.	
More precisely, they have introduced a new dual conformal invariant for any $(M^n,g)$ be a closed Riemannian manifold with $n\geqslant 5$.
Furthermore, they obtained 
\begin{equation}\label{aubinineqfourthreversed}
	\mathcal{Y}_4(M, g)=\mathcal{Y}_4^*(M, g) \leqslant \mathcal{Y}_4(\mathbb{S}^n,g_0).
\end{equation}	
Based on \eqref{aubinineqfourth}, by using classical bifurcation techniques on the universal cover of Schoen \cite{MR1173050}, Ratzkin \cite{arXiv:2002.05939} proved nonuniqueness results for fourth order $Q$-curvature metrics on the product $\mathbb S^1\times\mathbb S^{n-1}$. 
Almost concomitantly, Bettiol et al. \cite{MR4251294} applied a topological method to extend this construction for much broader situations; this depends on the reversed Aubin inequality in \eqref{aubinineqfourthreversed}.
These result are inspired by the classical works on the scalar curvature equation \cite{MR0125546,MR0431287,MR1173050}.
We also refer the reader to \cite{MR3504948,MR3803113,MR3105774,MR788292} for more result using a related bifurcation technique. 

%
In the light of \cite{MR3618119}, we speculate that both geometric inequalities above hold under milder conditions of the underlying manifold.

\begin{statement}{Conjecture~1}\label{conj1}
	Let $(M^n,g)$ be a closed Riemannian manifold with $n\geqslant7$.
	Assume that 
	\begin{equation}\label{hang-yang}\tag{$\mathcal{HY}$}
		\mathcal{Y}_2(M, g)\geqslant0, \quad \mathcal{Y}_4(M, g)\geqslant0, \quad {\rm and} \quad Q_6({g})\geqslant 0 \quad {\rm but} \quad Q_6({g})\not\equiv 0.
	\end{equation} 
	Then, \ref{itm:P} is satisfied. 
	In particular, both \ref{itm:A} and \ref{itm:A'} also hold.
\end{statement}
It is not hard to see that $(\mathbb S^1\times\mathbb S^{n-1},g_{\mathbb S^1}\oplus g_{\mathbb S^{n-1}})$ satisfies \eqref{hang-yang}. 
This fact somehow supports the conjecture above.
Consequently, all the methods to prove the existence of infinitely many conformal metrics with constant sixth order $Q$-curvature can be adapted promptly
\begin{statement}{Conjecture~2}
	Let $(M^n,g)=(C^{n_1}\times \mathcal{S}^{n_2},g_C\oplus g_{\mathcal{S}})$ be a complete noncompact Riemannian manifold with $n_1+n_2=n\geqslant 7$.
	Suppose that \eqref{hang-yang} holds.
	Assume that
	$(C^{n_1}, g_C)$ is a closed manifold with constant scalar curvature, and $(\mathcal{S}^{n_2}, g_{\mathcal{S}})$ is a simply-connected symmetric space of noncompact or Euclidean type.
	Then, there exists infinitely many nonhomothetic conformal periodic metrics on $\mathcal{M}_6(g_C\oplus g_{\mathcal{S}})$ with constant positive sixth order $Q$-curvature.
\end{statement}

Let us now describe the plan for the rest of the manuscript.
In Section~\ref{sec:notation}, we establish some notations that will be used throughout the paper.
In Section~\ref{sec:toymodel}, we use a bifurcation technique based on the jump of the Morse index on the universal cover to prove Theorem~\ref{thm1}.     
In Section~\ref{sec:noncompactcase}, we generalize the latter approach to the case of an infinite tower of Riemannian covers to give proof for Corollary~\ref{cor3}. 
In conclusion, we also explain how these main results are interconnected.

\section{Notation}\label{sec:notation}
Let us establish some standard terminology and definitions.
In what follows, we will always be using Einstein's summation convection.

\begin{itemize}
	\item $\delta=g_{\mathbb R^n}$ denotes the standard Euclidean metric;
	\item $g_0=g_{\mathbb S^n}$ denotes the standard round metric;
	\item  $(e_i)_{i=1}^n$ denotes a local coordinate frame;
	\item ${\rm Rm}_g\in\mathfrak{T}^3_1(M)$ denotes the {Riemannian curvature tensor},
	\item $\accentset{\circ}{\rm Rm}_g\in\mathfrak{T}^4_0(M)$ denotes the {covariant Riemann curvature tensor},
	\item ${\rm Ric}_g={\rm tr}_g\accentset{\circ}{\rm Rm}_g\in\mathfrak{T}^2_0(M)$ denotes the {Ricci curvature tensor} , which can be expressed as
	${\rm Ric}_{jk}=\accentset{\circ}{\rm Rm}_{i jk}^{i}=g^{i\ell} \accentset{\circ}{\rm Rm}_{ijk\ell}$;
	\item   $R_g={\rm tr}_g{\rm Ric}_g\in\mathfrak{T}^0_0(M)$ denotes the {scalar curvature} given by $R_g=g^{ij}{\rm Ric}_{ij}$;
	\item $\mathfrak{T}^r_s(M)$ denotes the set of $(r,s)$-type tensor over $M$  with $\mathfrak{T}^0_0(M)=\mathcal{C}^{\infty}(M)$;
	\item  $\Delta_{g}:=g^{i j} \nabla_{i} \nabla_{j}$ denotes the Laplace--Beltrami operator;
	\item $\nabla_g$ denotes the Levi--Civita connection;
	\item ${\rm tr}_g:\mathfrak{T}^r_s(M)\rightarrow\mathfrak{T}^{r-2}_s(M)$ denotes the trace operator.
	\item $a_1 \lesssim a_2$ if $a_1 \leqslant C a_2$, $a_1 \gtrsim a_2$ if $a_1 \geqslant C a_2$, and $a_1 \simeq a_2$ if $a_1 \lesssim a_2$ and $a_1 \gtrsim a_2$.
	\item $u=\mathcal{O}(f)$ as $x\rightarrow x_0$ for $x_0\in\mathbb{R}\cup\{\pm\infty\}$, if $\limsup_{x\rightarrow x_0}(u/f)(x)<\infty$ is the Big-O notation;
	\item $u=\mathrm{o}(f)$ as $x\rightarrow x_0$ for $x_0\in\mathbb{R}\cup\{\pm\infty\}$, if $\lim_{x\rightarrow x_0}(u/f)(x)=0$ is the little-o notation;
	\item $u\simeq\widetilde{u}$, if $u=\mathcal{O}(\widetilde{u})$ and $\widetilde{u}=\mathcal{O}(u)$ as $x\rightarrow x_0$ for $x_0\in\mathbb{R}\cup\{\pm\infty\}$;
	\item $\mathcal{C}^{j,\alpha}(M)$, where $j\in\mathbb N$ and $\alpha\in (0,1)$, is the classical H\"{o}lder space;  we simply denote $\mathcal{C}^{j}(M)$ when $\alpha=0$.
	\item $W^{j,q}(M)$ is the classical Sobolev space, where $j\in\mathbb N$ and $q\in[1,+\infty]$; when $j=0$ we simply denote $L^{q}(M)$ when $q=2$, we simply denote $H^{j}(M)$; 
	\item $\gamma_n=\frac{n-6}{2}$ is the Fowler rescaling exponent;
	\item $2^{\#}=\frac{2n}{n-6}$ is the critical Sobolev exponent.
\end{itemize}

It is also convenient to define some operations involving two tensors. 

\begin{definition}
	Let $(M^n,g)$ be a closed Riemannian manifold.
	We define the following operations with tensors
	\begin{itemize}
		\item[{\rm (a)}] {cross product} $\times:{\rm Sym}_{2}(M)\times{\rm Sym}_{2}(M)\rightarrow {\rm Sym}_{2}(M)$ given by
		\begin{equation*}
			(h_1 \times h_2)_{ij}:=g^{k\ell} h_{1,ik} h_{2,j\ell}=h_{1,i}^{\ell} h_{2,\ell j}.
		\end{equation*}	
		\item[{\rm (b)}] {dot product} $\cdot :{\rm Sym}_{2}(M)\times{\rm Sym}_{2}(M)\rightarrow \mathbb{R}$ given by
		\begin{equation*}
			h_1 \cdot h_2:={\rm tr}_g(h_1 \times h_2)=g^{ij} g^{k\ell} h_1^{ik} h_{2,j\ell}=h_1^{jk} h_{2,jk}.
		\end{equation*}	
		\item[{\rm (c)}] {Kulkarni--Nomizu product} $\owedge:{\rm Sym}_{2}(M)\times{\rm Sym}_{2}(M)\rightarrow \mathfrak{T}^4_0(M)$ given by
		\begin{equation*}
			(h_1\owedge h_2)_{ijk\ell}:=h_{1,i\ell}h_{2,jk}+h_{1,jk}h_{2,i\ell}-h_{1,ik} h_{2,j\ell}-h_{1,j \ell}h_{2,ik}.
		\end{equation*}	
		\item[{\rm (d)}]{dot operator} $\cdot:{\rm Sym}_{2}(M)\rightarrow{\rm Sym}_{2}(M)$ given by 
		\begin{equation*}
			(\accentset{\circ}{\rm Rm} \cdot h)_{j k}:=R_{ijk\ell} h^{i\ell}.
		\end{equation*}	
		\item[{\rm (e)}]{$L^{2}$-formal adjoint of the Lie derivative} $\delta_g:{\rm Sym}_{2}(M)\rightarrow\mathbb{R}$ given by
		\begin{equation*}
			\left(\delta_{g} h\right)_{i}:=-\left(\operatorname{div}_{g} h\right)_{i}=-\nabla_g^{j} h_{ij}.
		\end{equation*}	
	\end{itemize}
\end{definition}

\section{Toy model}\label{sec:toymodel}
This section is devoted to proving Theorem~\ref{thm1}.
Instead, we will prove a slightly stronger proposition based on the validity of Aubin's lemma in \ref{itm:A}.
Our proof follows the strategy in \cite{MR1173050} and \cite{arXiv:2002.05939} and relies strongly on the explicit construction of the Delaunay metrics on the cylinder from Proposition~\ref{prop:andrade-wei}.

\subsection{Variational structure}
We begin with the well-known variational characterization of constant $Q$-curvature metrics. 
let us recall the continuous classical Sobolev embedding 
\begin{equation}\label{sobolevembedding}
	W^{3,2}(M)\hookrightarrow L^{2^\#}(M),
\end{equation}
where $2^\#={2n}/{(n-6)}$ with $n\geqslant7$ is the critical Sobolev exponent.

\begin{remark}
	We emphasize that assuming that $(M^n,g)$ is Einstein with $n\geqslant 7$, the operator on the left-hand of \eqref{ourequationmanifold} has a much simpler form than \eqref{operators} $($c.f. \cite{MR3073887,MR2244375}$)$. 
	Namely, let us denote by $P_6(g):W^{3,2}(M)\rightarrow\mathbb R$ given by 
	\begin{align}\label{factorization}
		P_6(g)&:=(-\Delta_g+a_nR_g)(-\Delta_g+b_nR_g)(-\Delta_g+c_nR_g).
	\end{align}
	where $a_n,b_n,c_n>0$ given by
	\begin{align*}
		a_n:=\frac{(n-6)(n+4)}{4 n(n-1)}, \quad
		b_n:=\frac{(n-4)(n+2)}{4n(n-1)}, \quad {\rm and} \quad
		c_n:=\frac{n-2}{4(n-1)}
	\end{align*}
	are also dimensional constants. 
\end{remark}

\begin{definition}
	Let $(M^n,g)$ be a closed Riemannian Einstein manifold with $n\geqslant 7$.
	We define the bilinear form $\mathcal{E}_6(g):W^{3,2}(M)\times W^{3,2}(M)\rightarrow \mathbb R$ given by  
	\begin{equation*}
		\mathcal{E}_6(g)(u_1, u_2)=\displaystyle\int_M u_2 P_6(g)u_1 \ud \mathrm{v}_g,
	\end{equation*}     
	where $\ud \mathrm{v}_g$ denotes the standard volume measure.
	We set $\mathcal{E}_6(g)(u, u)=\mathfrak{Q}_6(g)(u)$ to be its associated quadratic form.
	More explicitly, the functional $\mathfrak{Q}_6(g):W^{3,2}(M)\rightarrow\mathbb R$ is defined as
	\begin{align}\label{quadraticform}
		\mathfrak{Q}_6(g)(u)&=-\int_M|\nabla_g \Delta_g u|^2 \ud\mathrm{v}_g+A_n\int_M |\Delta_g u|^2 \ud\mathrm{v}_g -B_n\int_M|\nabla_g u|^2 \ud\mathrm{v}_g+ C_n\int_M |u|^2 \ud\mathrm{v}_g.
	\end{align}
	with
	\begin{align}\label{coefficientspaneitseinstein}
		\nonumber
		A_n&:=\frac{3 n^2-6 n-32}{4 n(n-1)}\\
		B_n&:=\frac{3 n^4-12 n^3-52 n^2+128 n+192}{16 n^2(n-1)^2} \\\nonumber
		C_n&:=\frac{(n^4-20 n^2+64)(n-6)}{64 n^2(n-1)^3}
	\end{align}
	being dimensional constants. 
\end{definition}

\begin{remark}
	Notice that for any $u_1,u_2\in W^{3,2}(M)$, it is not hard to see that from \eqref{factorization} and the divergence theorem, it holds
	\begin{align*}
		\mathcal{E}_6(g)(u_1, u_2) =\int_M u_2 P_6(g) u_1 \ud \mathrm{v}_g=\int_M u_1 P_6(g) u_2 \ud \mathrm{v}_g=\mathcal{E}_6(g)(u_2, u_1),
	\end{align*}
	which shows that $\mathcal{E}_6(g):M\times M\rightarrow\mathbb R$ is a symmetric bilinear form.
	It is well-known the symmetry property together with the fact that the sixth order GJMS operator is self-adjoint holds in general for any metric, which dates back to the work \cite{MR1965361}.

\end{remark}

\begin{definition}
	Let $(M^n,g)$ be a closed Riemannian manifold with $n\geqslant7$.
	The total sixth order curvature functional $\mathcal{Q}_6:[g]\rightarrow\mathbb R$ is given by 
	\begin{equation}\label{totalcurvaturesixtorder}
		\mathcal{Q}_6(\bar{g})={{\rm vol}_{\bar{g}}(M)^{\frac{6-n}{n}}}{\int_M Q_{6}({\bar{g}}) \ud \mathrm{v}_{\bar{g}}}.
	\end{equation}
	In terms of conformal factors, we have 
	$\mathcal{Q}_6: W^{3,2}_+(M)\rightarrow\mathbb R$ is given by 
	\begin{equation}\label{yamabefunctionalsixtorder}
		\mathcal{Q}_6(u)= {\frac{2}{n-6}\|u\|^2_{L^{2^{\#}}(M)}}{\int_M u P_6(g)u \ud \mathrm{v}_g}={\frac{2}{n-6}\|u\|^2_{L^{2^{\#}}(M)}}\mathfrak{Q}_6(g)(u).
	\end{equation}
	is the sixth order Yamabe functional, where $W_{+}^{3,2}(M)$ denotes the subspace of $($almost everywhere$)$ positive functions in $W^{3,2}(M)$. 
\end{definition}

\begin{lemma}\label{lm:derivativeoftotalcurvature}
	Let $(M^n,g)$ be a closed Riemannian manifold with $n\geqslant7$.
	The sixth order Yamabe-type functional $\mathcal{Q}_6: W^{3,2}_+(M)\rightarrow\mathbb R$ given by \eqref{yamabefunctionalsixtorder}
	is $\mathcal{C}^1$ with $D\mathcal{Q}_6: W^{3,2}_+(M)\rightarrow\mathbb R$ given by 
	\begin{equation}\label{derivativeyamabe}
		D \mathcal{Q}_6(u)\phi=\frac{4}{n-6}\|u\|_{L^{2^{\#}}(M)} \int_M \phi\left(P_6(g)u-\|u\|_{L^{2^{\#}}(M)}^{-2^{\#}} \mathfrak{Q}_6(g)(u) u^{\frac{n+6}{n-6}}\right) \ud \mathrm{v}_g.
	\end{equation}
\end{lemma}

\begin{proof}
	Let $u \in W_{+}^{3,2}(M) \cap \mathcal{C}^0(M)$.
	Using the compactness of $(M^n,g)$, we can choose a test function $\phi \in W^{3,2}(M)$ such that $0<\sup |\phi|<2^{-1} \inf_{M} u$. 
	In particular, $u+t \phi \in W_{+}^{3,2}(M)$ for $0<t<3/ 2$. 
	On the one hand, by Taylor expanding \eqref{quadraticform} until second order, we obtain
	\begin{align}\label{variationalcharac1}
		\mathfrak{Q}_6(g)(u+t \phi) & =\mathfrak{Q}_6(g)(u)+2 t \mathcal{E}_6(g)(\phi,u)+t^2 \mathfrak{Q}_6(g)(\phi) \\\nonumber
		& =\int_M u P_6(g)u \ud \mathrm{v}_g+2 t \int_M \phi P_6(g)u \ud \mathrm{v}_g+t^2 \int_M \phi P_6(g)\phi \ud \mathrm{v}_g.
	\end{align}
	On the other hand, we have
	\begin{align}\label{variationalcharac2}
		\|u+t\phi\|_{L^{2^\#}(M)}^2
		&=\left(\int_M(u+t \phi)^{\frac{2 n}{n-6}} \ud\mathrm{v}_g\right)^{\frac{6-n}{n}}\\\nonumber
		&= \|u\|_{L^{2^\#}(M)}^{2}\left(\int_M \phi u^{\frac{n+6}{n-6}} \ud\mathrm{v}_g\right)+\mathcal{O}\left(t^2\|\phi\|_{W^{3,2}(M)}^2\right) \\\nonumber
		&=\|u\|_{L^{2^\#}(M)}^{-2}\left(1-2 t\|u\|_{L^{2^\#}(M)}^{-2^{\#}} \int_M \phi u^{\frac{n+6}{n-6}} \ud\mathrm{v}_g+\mathcal{O}\left(t^2\|\phi\|_{W^{3,2}(M)}\right)\right).
	\end{align}  
	Therefore, by combining \eqref{variationalcharac1} and \eqref{variationalcharac2}, we find
	\begin{align*}
		&\mathcal{Q}_6(u+t \phi)\\
		&=\frac{2}{(n-6)}\|u\|_{L^{2^\#}(M)}^{-2}\left[\mathfrak{Q}_6(g)(u)+2 t \int_M \phi\left(P_6(g)u-\mathfrak{Q}_6(g)(u)\|u\|_{L^{2^\#}(M)}^{-2^{\#}} u^{\frac{n+6}{n-6}}\right) \ud \mathrm{v}_g+\mathcal{O}(t^2\|\phi\|_{W^{3,2}(M)}^2)\right],
	\end{align*}
	which completes the proof.  
\end{proof}

\begin{lemma}\label{lm:characterizationofcrititcalpoints}
	Let $(M^n,g)$ be a closed Riemannian manifold with $n\geqslant7$.
	The total sixth order curvature $\mathcal{Q}_6:[g]\rightarrow\mathbb R$  given by \eqref{totalcurvaturesixtorder}
	is $\mathcal{C}^1$ with first variation $\ud\mathcal{Q}_6:T_{\tilde{g}}[g]\simeq \mathcal{C}^\infty(M)\rightarrow\mathbb R$ given by
	\begin{equation*}
		\ud \mathcal{Q}_6(\bar{g})\tilde{g}=\frac{n-6}{2}\int_M \phi\left(Q_6({\bar{g}})-\overline{Q}_6({\bar{g}})\right) \ud \mathrm{v}_g,
	\end{equation*}
	where $\overline{Q}_6({\bar{g}})$ denotes the spherical average of $Q_6({\bar{g}})$
	In particular, metric $\bar{g}=u^{{4}/{(n-6)}} g$ is a critical point of the functional $\mathcal{Q}_6:[g]\rightarrow\mathbb R$ if and only if its sixth order $Q$-curvature $Q_6({\bar{g}})$ is constant.
\end{lemma}

\begin{proof}
	Initially, assuming that $\bar{g}=u^{4/(n-6)}\in[g]$ is a critical point of $\mathcal{Q}_6$, one has
	\begin{equation*}
		\left.\frac{\ud}{\ud t}\right|_{t=0} \mathcal{Q}_6\left((u+t\phi)^{\frac{4}{n-6}} g\right)=0 \quad {\rm for \ each} \quad \phi\in W^{3,2}(M)
	\end{equation*}
	such that the norm $\|\phi\|_{L^{2^\#}(M)}\ll1$ is small enough.
	The last identity combined with \eqref{transformationlawoperator} yields
	\begin{equation*}
		P_6(g)u={\mathfrak{Q}_6(g)(u) u^{\frac{n+6}{n-6}}}{\|u\|_{L^{2^\#}(M)}^{-2^\#}}.
	\end{equation*}
	Hence, substituting this value into \eqref{derivativeyamabe} gives us 
	\begin{equation*}
		Q_6({\bar{g}})=\frac{2}{(n-6)} {\mathfrak{Q}_6(g)(u)}{\|u\|_{L^{2^\#}(M)}^{-2^{\#}}},
	\end{equation*}
	which is indeed a constant.
	
	Conversely, suppose $Q_6({\bar{g}})=\bar{Q}$ is a constant, in which case \eqref{transformationlawoperator} gives us
	\begin{equation*}
		P_6(g)u=\frac{n-6}{2} \bar{Q} u^{\frac{n+6}{n-6}} 
	\end{equation*}
	and so
	\begin{equation*}
		\mathfrak{Q}_6(g)(u)=\frac{n-6}{2} \bar{Q}\|u\|_{L^{2^\#}(M)}^{2^{\#}}.
	\end{equation*}
	Substituting this value into \eqref{derivativeyamabe} yields
	\begin{align*}
		\left.\frac{\ud}{\ud t}\right|_{t=0} \mathcal{Q}_6\left((u+t\phi)^{\frac{4}{n-6}} g\right) & =\frac{4}{(n-6)\|u\|_{L^{2^\#}(M)}^2} \int_M \phi\left(P_6(g)u-\frac{n-6}{2} \bar{Q}^{\frac{n+6}{n-6}}\right) \ud \mathrm{v}_g\\
		& =\frac{4}{(n-6)\|u\|_{L^{2^\#}(M)}^2} \int_M \phi\left(\frac{n-6}{2} \bar{Q} u^{\frac{n+6}{n-6}}-\frac{n-6}{2} \bar{Q} u^{\frac{n+6}{n-6}}\right) \ud \mathrm{v}_g\\
		& =0.
	\end{align*}
	From which, we conclude that $\bar{g}=u^{{4}/{(n-6)}} g$ is a critical point of $\mathcal{Q}_6$.
	
	The proof of the lemma is concluded.
\end{proof}

\begin{remark}
	By scale invariance, constant sixth order $Q$-curvature metrics can also be characterized as follows
	\begin{equation*}
		\|u\|_{L^{2^\#}(M)}=1.
	\end{equation*}
	It is easy to recover the constant $Q$-curvature equation \eqref{ourequationmanifold} as the Euler-Lagrange equation of this constrained variational problem and compute the value of the $Q$-curvature of $g=u^{{4}/{(n-6)}} g$ in terms of the Lagrange multiplier associated to the critical point $u$.
\end{remark}

\subsection{Cylindrical coordinates}   

This section is devoted to constructing a change of variables that transforms the local 
singular PDE \eqref{ourlocalPDE} problem into an ODE problem with constant coefficients.

Let $\mathbb{S}^n\hookrightarrow \mathbb R^{n+1}$ denote the $n$-dimensional sphere, which we equip with the standard round metric $g_0=g_{\mathbb S^n}$ given by the pullback of the standard Euclidean metric $\delta=g_{g_{\mathbb R^n}}$ under the stereographic projection $\Pi$.     
For any $k \in \mathbb{R}$ with $0\leqslant k\leqslant n$, we seek complete metrics on $\mathbb{S}^n \setminus\Lambda^k$ of the form $g = \mathrm{u}^{{4}/{(n-6)}} g_0$, where either $\Lambda^k\subset\mathbb S^n$ is a smooth $k$-dimensional submanifold (singular case) or $\Lambda=\varnothing$   (non-singular case).
In order to $g$ to be complete on $\mathbb{S}^n \setminus 
\Lambda^k$, one has to impose $\liminf_{\ud(p,\Lambda)} \mathrm{u}(p) = +\infty$. 
Furthermore,
we prescribe the resulting metric to have constant sixth order curvature, which we normalize as in \eqref{const_q_curv} to be $Q_6(g)=Q_6(g_0)\equiv Q_n$.
We focus on the case $k=0$, and so $\Lambda^0:=\{p_1,\dots,p_N\}$ for some $N\in\mathbb N$.
From now on, we assume $\Lambda^0:=\{p,-p\}$ with $p=\mathbf{e}_1$ the north pole and denote $\Lambda^0=\Lambda$.

In more analytical terms, the condition that $g = \mathrm{u}^{{4}/{(n-6)}} g_0$ satisfies $Q_6(g) = Q_n$ on 
$\mathbb{S}^n \setminus 
\Lambda$ is equivalent to the PDE 
\begin{equation}\tag{$\mathcal{Q}_{6,g_0}$}\label{ourequation}
	P_6({g_0})\mathrm{u}=c_n \mathrm{u}^{\frac{n+6}{n-6}} \quad \mbox{on} \quad \mathbb{S}^n \setminus \Lambda,
\end{equation}
where $c_n=\frac{n-6}{2}{Q}_n$ is a normalizing constant.    
The operator on the left-hand side is the sixth order GJMS operator on the sphere defined by 
\begin{equation*}
	P_6({g_0})=\left(-\Delta_{g_0}+\frac{(n-6)(n+4)}{4}\right)
	\left(-\Delta_{g_0}+\frac{(n-4)(n+2)}{4}\right)
	\left(-\Delta_{g_0}+\frac{n(n-2)}{4}\right).
\end{equation*}
In general, the GJMS operators are conformally covariant, and in particular
if $g=\mathrm{u}^{4/(n-6)}g_0$ is conformal to $g_0$ then
\begin{equation}\label{transformationlawpaneitz}
	P_6({g})\phi=\mathrm{u}^{-\frac{n+6}{n-6}}P_6({g_0})(\mathrm{u}\phi) \quad \mbox{for all} \quad \phi\in \mathcal{C}^{\infty}(\mathbb{S}^n \setminus 
	\Lambda).
\end{equation}
Again, we remark that the nonlinearity 
on the right-hand side of \eqref{ourequation} has critical growth with respect 
to the Sobolev embedding \eqref{sobolevembedding}. It is well known that this 
embedding is not compact, reflecting the conformal invariance of the PDE \eqref{ourequation}. 

It will be convenient to transfer the PDE \eqref{ourequation} to Euclidean 
space, which we can do using the standard stereographic projection (with the 
north pole in $\mathbb{S}^n \setminus 
\Lambda$, and in particular, a non-singular point of any of the 
metrics we consider). After stereographic projection, we can write 
\begin{equation*}
	g_0 = u_{\rm sph}^{\frac{4}{n-6}} \delta \quad {\rm with} \quad u_{\rm sph}
	(x) = \left ( \frac{1+|x|^2}{2} \right )^{\frac{6-n}{2}},
\end{equation*}
where $\delta$ is the Euclidean metric. 

In these coordinates, we have 
$g = \mathrm{u}^{{4}/{(n-6)}}g_0 = (\mathrm{u}\cdot u_{\rm sph})^{{4}/{(n-6)}}
\delta$. 
Thus, $u\in \mathcal{C}^{\infty}(\mathbb R^n\setminus\{0\})$ given by $u = \mathrm{u} \cdot u_{\rm sph}$ 
is a positive singular solution to the transformed equation 
\begin{flalign}\label{limitequationmanypunct}\tag{$\mathcal{Q}_{6,\delta}$}
	(-\Delta)^3u=c_{n}u^{\frac{n+6}{n-6}} \quad {\rm in} \quad \mathbb R^n\setminus\{0\},
\end{flalign}
where $\Delta$ is the usual flat Laplacian and  $\Pi(\Lambda):=\{0\}\subset\mathbb R^n$ is the image of the singular 
set $\Lambda\subset\mathbb S^n$ under the stereographic projection. 
As a notational shorthand, we adopt the 
convention that $\mathrm{u}$ refers to a conformal factor relating the metric $g$ to the 
round metric, {\it i.e.} $g = \mathrm{u}^{{4}/{(n-6)}} g_0$, while $u$ refers to a conformal 
factor relating the metric $g$ to the Euclidean metric, {\it i.e.} $g = u^{{4}/{(n-6)}}
\delta$, with the two related by $u = \mathrm{u} \cdot u_{\rm sph}$. 	
In this Euclidean setting, the transformation law \eqref{transformationlawpaneitz} in particular 
implies the scaling law for \eqref{limitequationmanypunct}, namely if $u$ solves 
\eqref{limitequationmanypunct} then so does the rescaling
\begin{equation*}
	u_{\lambda}(x):=\lambda^{\frac{n-6}{2}}u(\lambda x)
\end{equation*}
for any $\lambda >0$.

\begin{definition}\label{def:cylindricaltransformation}
	We define the sixth order, a logarithmic cylindrical change of variables as $\mathfrak{F}: \mathcal{C}^\infty (\mathbb B_R^*) \rightarrow \mathcal{C}^\infty
	(\mathcal{C}_T)$ given by
	\begin{equation} \label{cylindricaltransformation}
		\mathfrak{F}(u) (t,\theta) = e^{\frac{6-n}{2} t} 
		u(e^{-t} \theta) = v(t,\theta) \quad {\rm with} \quad \theta=x/|x|.
	\end{equation} 
	where $R>0$ and $T= -\ln R$ and $\mathcal{C}_T = (T,+\infty) \times 
	\mathbb{S}^{n-1}$. 
	We also set its the inverse transform $\mathfrak{F}^{-1} : \mathcal{C}^\infty(\mathcal{C}_T) 
	\rightarrow \mathcal{C}^\infty (\mathbb B_R^*)$ is given by 
	\begin{equation*}
		\mathfrak{F}^{-1} 
		(v)(x) = |x|^{\frac{6-n}{2}} v(-\ln |x|, \theta) = u(x).
	\end{equation*}
	The exponent $\gamma_n:=\frac{n-6}{2}$ is sometimes called Fowler rescaling exponent.
\end{definition} 

Using this change of variables, we arrive at the following 
sixth order nonlinear PDE on the cylinder 
\begin{equation}\tag{$\mathcal C_{T}$}\label{ourPDEcyl}
	-P_{\rm cyl}v=c_{n}v^{\frac{n+6}{n-6}} \quad {\rm on} \quad {\mathcal{C}}_T.
\end{equation}
Here $P_{\rm cyl}$ is the sixth order GJMS operator associated to 
the cylindrical metric $g_{\rm cyl} = \ud t^2 + \ud\theta^2$ on $\mathcal{C}_{\infty}:=\mathbb{R} \times 
\Ss^{n-1}$, and it is given by
\begin{align}\label{cylindricalpaneitz}
	P_{\rm cyl}:=P_{\rm rad}+P_{\rm ang},
\end{align}
where
\begin{align*}
	P_{\rm rad}:=\partial_t^{(6)}-K_{4}\partial_t^{(4)}+K_{2}\partial_t^{(2)}-K_{0}
\end{align*}
and
\begin{align*}
	P_{\rm ang}:=2\partial^{(4)}_t\Delta_{\theta}-J_3\partial^{(3)}_t\Delta_{\theta}+J_2\partial^{(2)}_t\Delta_{\theta}-J_1\partial_t\Delta_{\theta}+J_0\Delta_{\theta}+3\partial^{(2)}_t\Delta^2_{\theta}-L_0\Delta^2_{\theta}+\Delta^3_{\theta}
\end{align*}
with
\begin{align}\label{coefficients}
	&\nonumber K_{0}=2^{-8}(n-6)^2(n-2)^2(n+2)^2&\\\nonumber 
	&K_{2}=2^{-4}(3n^4-24n^3+72n^2-96n+304)&\\\nonumber
	&K_{4}=2^{-2}(3n^2-12n+44)&\\
	&J_{0}=2^{-3}(3n^4-18n^3-192n^2+1864n-3952)&\\\nonumber
	&J_{1}=2^{-1}(3n^3+3n^2-244n+620)&\\\nonumber
	&J_{2}=2 n^2+13n-68&\\\nonumber
	&J_{3}=2 (n+1)&\\\nonumber
	&L_{0}=2^{-2}(3 n^2-12n-20)&
\end{align}
being dimensional constants.

\begin{remark}
	The following decomposition holds
	\begin{equation*}
		P_{\rm rad}=L_{\lambda_1}\circ L_{\lambda_2}\circ L_{\lambda_3},
	\end{equation*}
	where $L_{\lambda_j}:=-\partial_t^2+\lambda_j$ for $j=1,2,3$ with
	\begin{equation*}
		\lambda_1=\frac{n-6}{2}, \quad  \lambda_2=\frac{n-2}{2}, \quad {\rm and} \quad \lambda_3=\frac{n+2}{2}.
	\end{equation*}
	For the proof, we refer the reader to \cite[Proposition~2.7]{arXiv:2210.04376}.
\end{remark}

\subsection{Delaunay metrics}
This section presents some particular model metrics on the moduli space. 
Let $\Lambda=\{p_1,p_2\}\subset\mathbb S^n$, which without loss of generality can be chosen such 
that $p_1=\mathbf{e}_n$ is the north pole, and $p_2=-p_1$ is the south pole.
The conformal factor $\mathrm{u}:\mathbb{S}^{n} \backslash \{ p_1, p_2\} \rightarrow (0,+\infty)$ 
determines a metric $g\in\mathcal{M}_{6,\Lambda}$ and after  composing with 
a stereographic projection it corresponds to a singular solution to \eqref{ourlimitPDE}.

Let us start with  the conformally flat equation  
\begin{flalign}\tag{$\mathcal P_{6,R}$}\label{ourlocalPDE}
	(-\Delta)^3u=c_{n}u^{\frac{n+6}{n-6}} \quad {\rm in} \quad \mathbb{B}^*_R,
\end{flalign}
where $\mathbb{B}_R^*:=\{x\in\mathbb{R}^n : 0<|x|<R\}$ is the punctured ball for $R<+\infty$. 

Allowing 
$R\rightarrow+\infty$ turns \eqref{ourlocalPDE} into the following PDE on the punctured space
\begin{flalign}\tag{$\mathcal P_{6,\infty}$}\label{ourlimitPDE}
	(-\Delta)^3u=c_{n}u^{\frac{n+6}{n-6}} \quad {\rm in} \quad \mathbb R^n\setminus\{0\}.
\end{flalign}
On this subject, the classification of non-singular solutions to \eqref{ourlimitPDE} is provided in \cite{MR1679783}.
Later on, in \cite{arxiv:1901.01678}, it is proved that blow-up limit solutions exist.
Finally, based on a topological shooting method, the first and last authors recently classified 
all possible solutions to this limit equation \cite{arXiv:2210.04376}.

One can merge these classification	results into the statement below
\begin{propositionletter}\label{prop:andrade-wei}
	Let $u\in\mathcal{C}^6(\mathbb R^n\setminus\{0\})$ be a positive solution to \eqref{ourlimitPDE}. 
	Assume that
	\begin{enumerate}
		\item[{\rm (a)}] the origin is a removable singularity, then there exists $x_0\in\mathbb{R}^n$ and $\varepsilon>0$ such that $u$ is radially symmetric about $x_0$ and, up to a constant, is given by 
		\begin{equation}\label{sphericalsolutions}
			u_{x_0,\varepsilon}(x)=\left(\frac{2\varepsilon}{1+\varepsilon^{2}|x-x_0|^{2}}\right)^{\frac{n-6}{2}}.
		\end{equation}
		These are called the {\it $($sixth order$)$ spherical solutions} $($or bubbles$)$.
		\item[{\rm (b)}]  the origin is a non-removable singularity, then $u$ is radially symmetric about the origin. Moreover, there exist $\varepsilon_0 \in (0,\varepsilon_*]$ and $T\in (0,T_{\varepsilon_0}]$ such that
		\begin{equation*}
			u_{\varepsilon,T}(x)=|x|^{\frac{6-n}{2}}v_{\varepsilon}(-\ln|x|+T).
		\end{equation*}
		Here $\varepsilon_*=K_0^{(n-6)/6}$, $T_\varepsilon\in\mathbb{R}$ is the fundamental period of the unique $T$-periodic bounded solution $v_T$ to the following sixth order IVP, 
		\begin{equation*}
			\begin{cases}
				v^{(6)}-K_4v^{(4)}+K_2v^{(2)}-K_0v=c_nv^{\frac{n+6}{n-6}}\\
				v(0)=\varepsilon_0,\ v^{(2)}(0)=\varepsilon_2,\ v^{(4)}(0)=\varepsilon_4,\ v^{(1)}(0)=v^{(3)}(0)=v^{(5)}(0)=0,
			\end{cases}
		\end{equation*}
		where $K_4,K_2,K_0,\varepsilon_*>0$ are dimensional constants $\varepsilon_0\in (0,\varepsilon_*]$ $($See \eqref{coefficients} and \eqref{constantsolutions}$)$. 
		These are called $($sixth order$)$ Emden--Fowler solutions.
	\end{enumerate}  
\end{propositionletter}

Applying the cylindrical transform \eqref{cylindricaltransformation} to this 
PDE, in turn, yields 
\begin{equation}\tag{$\mathcal C_{\infty}$}\label{ourPDEcyllimit}
	-P_{\rm cyl}v=c_{n}v^{\frac{n+6}{n-6}} \quad {\rm on} \quad {\mathcal{C}}^n_\infty.
\end{equation}
where we recall $\mathcal{C}_\infty:=\mathbb R\times\mathbb S^n$ is the cylinder.
Next, using that all solutions to \eqref{ourequation} are radially symmetric with respect 
to the origin, \eqref{ourPDEcyllimit} reduces to a sixth order ODE problem
\begin{equation}\tag{$\mathcal{O}_{6,\infty}$}\label{ourODE}
	-v^{(6)}+K_4v^{(4)}-K_2v^{(2)}+K_0v=c_nv^{\frac{n+6}{n-6}} \quad {\rm in} \quad \mathbb R.
\end{equation}

From this last formulation, one can quickly find two equilibrium solutions.
First, the cylindrical solution 
\begin{equation}\label{cylindricalsolution}
	v_{\rm cyl}(t)\equiv \varepsilon_* > 0,
\end{equation}
which is the only constant solution, where $\varepsilon_*>0$ is given by \eqref{constantsolutions}. 
Second, the spherical solution
\begin{equation}\label{sphericalsolution}
	v_{\rm sph}(t,\theta)=\cosh (t)^{\frac{6-n}{2}}.
\end{equation}
Undoing the cylindrical change of variables, we have
\begin{equation*}
	u_{\rm sph}(x)=\left(\frac{1+|x|^2}{2} \right )^{\frac{6-n}{2}}
\end{equation*}
which is non-singular at the origin and is a particular case of \eqref{sphericalsolutions} with $\varepsilon = 1$ 
and $x_0 = 0$. 
Also, for the cylindrical solution, it follows
\begin{equation*}
	u_{\rm cyl}(x)=\varepsilon_*|x|^{\frac{6-n}{2}},
\end{equation*}
which is singular at the origin.

In this setting,
Proposition~\ref{prop:andrade-wei} classifies all positive solutions 
$v_{\varepsilon_0} \in \mathcal{C}^{6}(\mathbb{R})$ to \eqref{ourODE} in terms of the necksize 
$\varepsilon_0 \in(0, \varepsilon_{*}]$, where 
$\varepsilon_0=\min_{\mathbb R}v \in(0, \varepsilon_{*}]$.
Henceforth, we keep $\varepsilon_0=\varepsilon$ for simplicity.
Varying the parameter $\varepsilon$ from its maximal value of $\varepsilon_*$ to $0$, 
we see that the Delaunay solutions in Proposition~\ref{prop:andrade-wei} (b) interpolate 
between the cylindrical solution $v_{\rm cyl}$ and the spherical solution $v_{\rm sph}$. Therefore, we 
denote the minimal period of $v_\varepsilon$ by $T_\varepsilon$. 
In other words, we have that entire solutions to \eqref{ourODE} are either: the constant cylindrical solution $v_{\rm cyl}$, translates of the spherical solution $v_{\rm sph}$, or translates of a Delaunay solution $v_{\varepsilon}$ for some $\varepsilon \in(\varepsilon_*, 1)$.

\begin{lemmaletter}\label{lm:classification}
	Let $v \in \mathcal{C}^{6}(\mathbb{R})$ be a solution to \eqref{ourODE}.
	One of the following three alternatives holds:
	\begin{itemize}
		\item[{\rm (a)}] $v\equiv \pm \varepsilon_{*}$ or $v \equiv 0$;
		\item[{\rm (b)}] $v(t)=\pm c_{n}(2\cosh (t-T))^{\frac{6-n}{2}}$ for some $T \in \mathbb{R}$;
		\item[{\rm (c)}] $v(t)=v_\varepsilon(t-T)$ for some $\varepsilon>0\in(0,\varepsilon_{*})$ on is a periodic solution.
	\end{itemize}
\end{lemmaletter}

\begin{proof}
	See \cite[Lemma 5.1]{arXiv:2210.04376}.
\end{proof}

\begin{definition}
	The conformal metrics 
	\begin{equation*}
		g_{\rm sph}=u_{\rm sph}^{\frac{4}{n-6}}\delta \quad {\rm and} \quad  g_{\rm cyl}=u_{\rm cyl}^{\frac{4}{n-6}}\delta
	\end{equation*}
	will be called spherical and cylindrical metrics, respectively.
	For each $\varepsilon\in(0,\varepsilon_*)$, the conformal metric given by
	\begin{equation*}
		g_{\varepsilon}=u_{\varepsilon}^{\frac{4}{n-6}}\delta 
	\end{equation*}
	will be called a Delaunay metric. 
	The family of Delaunay metrics, denoted by $\{g_\varepsilon\}_{\varepsilon\in(0,\varepsilon_*)}$ interpolates between the spherical metric $g_0=g_{\rm sph}$ and the cylindrical metric $g_{\varepsilon_*}=g_{\rm cyl}$.
\end{definition}

\subsection{Hamiltonian energy}
We now turn to a discussion of the existence and specific form of a family of homological integral 
invariants of solutions of equation \eqref{ourequation}. 
Let us start with the classical definition of Hamiltonian energy for a solution $v\in \mathcal{C}^{6}(\mathcal{C}_{\infty})$ to \eqref{ourPDEcyllimit}.
\begin{definition}
	For any real function $v\in \mathcal{C}^{6}(\mathcal{C}_{\infty})$, let us define its Hamiltonian energy $($with respect to \eqref{ourPDEcyllimit}$)$ $ \mathcal{H}:\mathcal{C}_{\infty}\times \mathcal{C}^{6}(\mathcal{C}_{\infty})\rightarrow\mathbb R$ by
	\begin{equation*}
		\mathcal{H}_{\rm cyl}(v):= \mathcal{H}_{\rm rad}(v)+ \mathcal{H}_{\rm ang}(v)+F(v),
	\end{equation*}
	where
	\begin{equation*}
		\mathcal{H}_{\rm rad}(v):=\frac{1}{2}{v^{(3)}}^2+\frac{K_4}{2}{v^{(2)}}^{2}+\frac{K_2}{2}{v^{(1)}}^2-\frac{K_0}{2}v^2
		+v^{(5)}v^{(1)}-v^{(4)}v^{(2)}-{K_4}
		v^{(3)}v^{(1)},
	\end{equation*}
	is the radial part,
	\begin{align*}
		\mathcal{H}_{\rm ang}(v)&:=-J_4\left(\partial_t^{(3)}\nabla_\theta v\partial_t\nabla_\theta v-|\partial_t^{(2)}\nabla_\theta v|^2\right)-\frac{J_2}{2}|\partial_t^{(2)}\nabla_\theta v|^2-\frac{J_1}{2}|\partial_t^{(2)}\nabla_\theta v|^2&\\
		&-\frac{J_0}{2}|\nabla_\theta v|^2+\frac{L_2}{2}|\partial_t^{(2)}\Delta_\theta v|^2+\frac{L_0}{2}|\partial_t^{(2)}\Delta_\theta v|^2+\frac{1}{2}|\Delta_{\theta}v|^2&
	\end{align*}
	is the angular part, and
	\begin{equation*}
		F(v):=\frac{c_n(n-6)}{2n}|v|^{\frac{2n}{n-6}}
	\end{equation*}
	is the nonlinear term.
\end{definition}

With this definition in hands, it is direct to prove that this energy is conserved along with solutions to \eqref{ourODE}, that is, the following identity holds
\begin{equation*}
	\dfrac{\ud}{\ud t} \mathcal{H}(t,v)\equiv0.
\end{equation*}
In other terms, there exists $\mathcal H_{v}\in\mathbb R$ such that
\begin{equation*}
	\mathcal{H}(v)(t)= \mathcal{H}(t,v)\equiv \mathcal{H}(v):= \mathcal{H}_{v}.
\end{equation*}

\begin{remark}
	Notice that setting
	\begin{equation*}
		G(v):=F(v)-\frac{K_0}{2}v^2
	\end{equation*}
	the real roots of $G$ are precisely the only equilibrium solutions to \eqref{ourODE}, namely
	\begin{equation*}
		v\equiv0 \quad {\rm and} \quad v\equiv\pm \varepsilon_{*},
	\end{equation*}
	where
	\begin{equation}\label{constantsolutions}
		\varepsilon_{*}:=K_0^{\frac{n-6}{12}}=\left(\frac{(n-6)(n-2)(n+2)}{8}\right)^{\frac{n-6}{6}}.
	\end{equation}
	is a dimensional constant.
\end{remark}

\begin{remark}
	By computing this energy on the cylindrical and spherical solutions, we find
	\begin{equation*}
		\mathcal{H}_{\rm sph}:= \mathcal{H}\left(v_{\rm s p h}\right)=0 \quad {\rm and} \quad  \mathcal{H}_{\rm cyl}:= \mathcal{H}\left(v_{\rm c y l}\right)<0.
	\end{equation*}
	Furthermore, we see that the level set 
	\begin{equation*}
		\{ \mathcal{H}\equiv0\} \cap\{v^{(1)}=0\}\cap \{v^{(3)}=0\}\cap \{v^{(5)}=0\}
	\end{equation*}
	consists entirely of the (homoclinic) solution curve of the spherical solution and the origin.
	Whereas, for each $H\in(0,-\mathcal{H}_{\rm cyl})$
	the level set 
	\begin{equation*}
		\{ \mathcal{H}\equiv H\} \cap\{v^{(1)}=0\}\cap \{v^{(3)}=0\}\cap \{v^{(5)}=0\}
	\end{equation*}
	is a closed curve associated to the Delaunay solution $v_\varepsilon$ for some $\varepsilon \in(\varepsilon_*, 1)$. 
\end{remark}

Our last lemma will be referred to as the strong comparison result to reflect that for two bounded solutions to be equal, they only need to coincide until the first order at the origin.
Now we prove that the Hamiltonian energy is a parameter that orders bounded solutions in the $(v, v^{(1)})$-phase plane.

\begin{lemmaletter}\label{lm:energyordering}
	Let $v_1,v_2\in \mathcal{C}^{6}(\mathbb{R})$ be bounded solutions to \eqref{ourODE}.
	Suppose that $v_1(0)=v_2(0)$.
	The following holds:
	\begin{itemize}
		\item[{\rm (i)}] If $v_1^{(1)}(0)>v_2^{(1)}(0) \geqslant 0$  or $v_1^{(1)}(0)<v_2^{(1)}(0) \leqslant 0$,  then $ \mathcal{H}(v_1)> \mathcal{H}(v_2)$. 
		\item[{\rm (ii)}] If $v_1^{(1)}(0)<v_2^{(1)}(0) \leqslant 0$  or $v_2^{(1)}(0)>v_1^{(1)}(0) \geqslant 0$,  then $ \mathcal{H}(v_1)< \mathcal{H}(v_2)$.
		\item[{\rm (iii)}]  If $v_1(0)=v_2(0)$,  then $ \mathcal{H}(v_1)= \mathcal{H}(v_2)$.
	\end{itemize}
	In particular, we see 
	\begin{equation*}
		\lim _{\varepsilon \nearrow 1}  \mathcal{H}(\varepsilon)=0,
	\end{equation*}
	where $ \mathcal{H}(\varepsilon):= \mathcal{H}(v_{\varepsilon})$ is the Hamiltonian energy of the associated Emden--Fowler solution.
\end{lemmaletter}

\begin{proof}
	See \cite[Lemmas 4.20 and 4.21]{arXiv:2210.04376}.
\end{proof}

\subsection{Linearized operator}
Now we study the linearized operator around Delaunay solutions. The heuristics are that when this operator is Fredholm, its indicial roots determine the rate at which singular solutions to the nonlinear problem \eqref{ourPDEcyl} converge to this limit solution near the isolated singularity. 
Here, we borrow some ideas from  \cite{MR1936047,MR1356375}. 

First, let us consider the following nonlinear operator $\mathcal{N}_{\rm cyl}: H^{6}(\mathcal{C}_T)\rightarrow L^{2}(\mathcal{C}_T)$ given by
\begin{equation*}
	\mathcal{N}_{\rm cyl}(v):=-P_{\rm cyl}v+c_nv^{\frac{n+6}{n-6}}.
\end{equation*}
Then, we compute its linearization around Delaunay solutions classified above, as follows

\begin{lemma}\label{lm:linearizedoperator}
	The linearization of $\mathcal{N}_{\rm cyl}: H^{6}(\mathcal{C}_T)\rightarrow L^{2}(\mathcal{C}_T)$ around an Emden--Fowler solution $v_{\varepsilon,T}$ is given by 
	\begin{equation}\label{linearization}
		\mathscr{L}_{\rm cyl}[v_{\varepsilon}](\psi):=-P_{\rm cyl}v+\hat{c}_nv_{\varepsilon}^{\frac{12}{n-6}}\psi,
	\end{equation}
	where $\hat{c}_n:=\frac{n+6}{n-6}c_n$.
	From now on, we simply denote $\mathscr{L}_{\rm cyl}[v_{\varepsilon}](\psi):=\mathscr{L}_{\rm cyl}^{\varepsilon}(\psi)$.
	Notice that $\mathscr{L}_{\rm cyl}^\varepsilon(\psi)=\ud\mathcal{N}_{\rm cyl}[v_{\varepsilon}](\psi)$ is the Fr\'echet derivative of the nonlinear functional $\mathcal{N}_{\rm cyl}$.
\end{lemma}

\begin{proof} 
	By definition, we have 
	\begin{align}\label{linearizedcylvect}
		\mathscr{L}_{\rm cyl}[v_{\varepsilon}](\psi)=\frac{\ud}{\ud t}\Big|_{t=0}\mathcal{N}_{\rm cyl}(v_{\varepsilon}+t\psi).
	\end{align}
	Next, by routine computations, it follows 
	\begin{align*}
		\mathcal{N}_{\rm cyl}(v_{\varepsilon}+t\psi)-\mathcal{N}_{\rm cyl}(v_{\varepsilon})
		&=-P_{\rm cyl}v_{\varepsilon}+tP_{\rm cyl}\psi-(v_{\varepsilon}+t\psi)^{\frac{n+6}{n-6}}+P_{\rm cyl}v_{\varepsilon}+v_{\varepsilon}^{\frac{n+6}{n-6}}&\\
		&=tP_{\rm cyl}\psi-(v_{\varepsilon}+t\psi)^{\frac{n+6}{n-6}}+v_{\varepsilon}^{\frac{n+6}{n-6}},			
	\end{align*}
	which yields
	\begin{align*}
		\mathcal{N}_{\rm cyl}(v_{\varepsilon}+t\psi)-\mathcal{N}_{\rm cyl}(v_{\varepsilon})=tP_{\rm cyl}\psi-tc_n\left(-P_{\rm cyl}v+\frac{n+6}{n-6}c_nv_{\varepsilon}^{\frac{12}{n-6}}\psi\right)+\mathcal{O}(t^2) \quad {\rm as} \quad t\rightarrow 0
	\end{align*}
	which by taking $t\rightarrow0$ implies \eqref{linearizedcylvect}.
\end{proof}

Unfortunately, the linearized operator is not generally Fredholm since it has no closed range \cite[Theorem~5.40]{MR1348401}. 
This issue is caused by nontrivial elements on its kernel; these are called the Jacobi fields \cite{MR1936047}.  
Therefore, we must introduce suitable weighted Sobolev and H\"{o}lder spaces on which the linearized operator has a well-defined right-inverse, up to a discrete set on the complex plane.
For more details, see \cite[Section~2]{MR1763040}.

\begin{definition}
	Given $j,q\geqslant1$ and $\beta\in\mathbb{R}$, for any $v \in L_{\loc}^{q}(\mathcal{C}_T)$ define the following weighted Lebesgue norm
	\begin{equation*}
		\|v\|_{L_{\beta}^{q}(\mathcal{C}_T)}^{q}=\int_{0}^{\infty} \int_{\mathbb{S}^{n-1}}e^{-2\beta t}|v(t, \theta)|^{q} \ud\theta\ud t.
	\end{equation*}	
	Also, let us define the weighted Lebesgue space by 
	\begin{equation*}
		L_{\beta}^{q}(\mathcal{C}_T)=\left\{v\in L_{\loc}^{q}(\mathcal{C}_T) : \|v\|_{L_{\beta}^{q}(\mathcal{C}_T)}<\infty\right\}.
	\end{equation*}
	Similarly, consider the Sobolev spaces $W_{\beta}^{k,q}(\mathcal{C}_T)$ of functions with $k$ weak derivatives in $L^{q}$ having finite weighted norms.
	Here we also denote the Hilbert spaces
	\begin{equation*}
		W^{j,2}_{\beta}(\mathcal{C}_T)=H^{j}_{\beta}(\mathcal{C}_T) \quad {\rm and} \quad W^{j,q}(\mathcal{C}_T)=W^{j,q}(\mathcal{C}_T).
	\end{equation*}
	Notice that when $\beta=0$, we recover the classical Sobolev spaces. 
\end{definition}

\begin{definition}
	Given $j\geqslant1$, $\beta\in\mathbb{R}$ and $\zeta\in(0,1)$,
	for any $u \in \mathcal{C}_{\loc}^{0,\beta}(\mathcal{C}_T)$ define the following norm 
	\begin{equation*}
		\|v\|_{\mathcal{C}_{\beta}^{0,\zeta}(\mathcal{C}_T)}=\sup_{T>1} \sup \left\{\frac{e^{-\beta t_{1}} |v(t_{1}, \theta_{1})|-e^{-\beta t_{2}}|v(t_{2}, \theta_{2})|}{\ud_{\rm cyl}((t_{1}, \theta_{1}),(t_{2}, \theta_{2}))^{\zeta}}:(t_{1}, \theta_{1}),(t_{2}, \theta_{2}) \in\mathcal{C}_{T-1,T+1}\right\}.
	\end{equation*}
	For $m=0$, we weighted Hold\"{e}r space by 
	\begin{equation*}
		\mathcal{C}_{\beta}^{0,\zeta}(\mathcal{C}_T)=\left\{v\in \mathcal{C}_{\loc}^{0,\beta}(\mathcal{C}_T) : \|v\|_{\mathcal{C}_{\loc}^{0,\beta}(\mathcal{C}_T)}<\infty\right\}.
	\end{equation*}
	One can similarly define higher order weighted Hold\"{e}r spaces $\mathcal{C}_{\beta}^{m,\zeta}(\mathcal{C}_T)$ for $m\geqslant 1$.
\end{definition}

\begin{definition}\label{def:jacobifield}
	The {Jacobi fields} in the kernel of $\mathscr{L}_{\rm cyl}^{\varepsilon}:H^{6}_{\beta}(\mathcal{C}_T)\rightarrow L^{2}_{\beta}(\mathcal{C}_T)$, are the solutions $\psi\in H^{6}_{\beta}(\mathcal{C}_T)$ of the following sixth order linear equation,
	\begin{equation*}
		\mathscr{L}_{\rm cyl}^\varepsilon(\psi)=0 \quad {\rm on} \quad \mathcal{C}^n_T.
	\end{equation*}
\end{definition}

\subsection{Fourier eigenmodes}
We study the kernel of a linearized operator around a Delaunay solution by decomposing into its { Fourier eigenmodes}, a { separation of variables} technique. 

First, let us consider $\{\lambda_j,\chi_j(\theta)\}_{j\in\mathbb{N}}$ the eigendecomposition of the Laplace--Beltrami operator on $\mathbb{S}^{n-1}$ with the normalized eigenfunctions, 
\begin{equation}\label{spectrallaplacian}
	\Delta_{\theta}\chi_j(\theta)=-\lambda_{j}\chi_j(\theta).
\end{equation}
Here the eigenfunctions $\{\chi_j(\theta)\}_{j\in\mathbb{N}}$ are called spherical harmonics with associated sequence of eigenvalues $\{\lambda_j\}_{j\in\mathbb{N}}$ given by $\lambda_j=j(j+n-2)$ counted with multiplicity $\mathfrak{m}_j$, which are defined by  
\begin{equation*}
	\mathfrak{m}_0=1 \quad \mbox{and} \quad \mathfrak{m}_j=\frac{(2j+n-2)(j+n-3)!}{(n-2)!j!}.
\end{equation*}
In particular, we have $\lambda_0=0, \quad \lambda_1=\dots=\lambda_n=n-1$, $\lambda_{j}\geqslant 2n$, if $j>n$ and $\lambda_j\leqslant\lambda_{j+1}$.
Moreover, these eigenfunctions are the restrictions to $\mathbb{S}^{n-1}$ of homogeneous harmonic polynomials in $\mathbb{R}^n$. Here we denote by $V_j$ the eigenspace spanned by ${\chi_j(\theta)}$. 
Next, using \eqref{spectrallaplacian}, it is easy to observe 
\begin{equation}\label{spectralbilaplacian}
	\Delta^2_{\theta}\chi_j(\theta)=\lambda_j^2 \chi_j(\theta) \quad {\rm and} \quad \Delta^3_{\theta}\chi_j(\theta)=-\lambda_j^3 \chi_j(\theta).
\end{equation}
is the spectral data of $\Delta^2_{\theta}$ and $-\Delta^3_{\theta}$, respectively.

Now, using the expression \eqref{cylindricalpaneitz} in the formula \eqref{linearization}, we get
\begin{align*}
	\mathscr{L}_{\rm cyl}^{\varepsilon}(\psi)&:=\psi^{(6)}-K_{4}\psi^{(4)}+K_{2}\psi^{(2)}-K_{0}\psi-\hat{c}_nv_{\varepsilon}^{\frac{12}{n-6}}\psi\\
	&+2\Delta_{\theta}\psi^{(4)}-J_3\Delta_{\theta}\psi^{(3)}+J_2\Delta_{\theta}\psi^{(2)}-J_1\Delta_{\theta}\psi^{(1)}+J_0\Delta_{\theta}\psi+3\Delta^2_{\theta}\psi^{(2)}-L_0\Delta^2_{\theta}\psi+\Delta^3_{\theta}\psi.
\end{align*}
It is convenient to rearrange the terms as follows
\begin{align*}
	\mathscr{L}_{\rm cyl}^{\varepsilon}(\psi)&:=\psi^{(6)}-K_{4}\psi^{(4)}+2\Delta_{\theta}\psi^{(4)}+K_{2}\psi^{(2)}+J_2\Delta_{\theta}\psi^{(2)}-J_3\Delta_{\theta}\psi^{(3)}+3\Delta^2_{\theta}\psi^{(2)}-J_1\Delta_{\theta}\psi^{(1)}\\
	&-K_{0}\psi+J_0\Delta_{\theta}\psi-L_0\Delta^2_{\theta}\psi+\Delta^3_{\theta}\psi-\hat{c}_nv_{\varepsilon}^{\frac{12}{n-6}}\psi.
\end{align*}
Next, using \eqref{spectrallaplacian} and \eqref{spectralbilaplacian}, we can
project this operator on the eigenspaces, which gives us 
\begin{align}\label{jacobiequationproj}
	\mathscr{L}_{\rm cyl}^{\varepsilon,j}(\psi)&=\psi^{(6)}+(-K_4+2\lambda_j)\psi^{(4)}+\lambda_jJ_3\psi^{(3)}+(K_2+2J_2\lambda_j)\psi^{(2)}+\lambda_jJ_1\psi^{(1)}\\\nonumber
	&+\left(-K_0+J_0\lambda_j-L_0\lambda_j^2+\lambda_j^3-\hat{c}_nv_{\varepsilon}^{\frac{12}{n-6}}\right)\psi \quad {\rm for} \quad j\in\mathbb{N}.
\end{align}
Moreover, for any $\phi\in L^{2}(\mathbb{S}^{n-1})$, we write
\begin{equation*}
	\phi(t,\theta)=\sum_{j\in\mathbb N}\phi_j(t)\chi_j(\theta), \quad \mbox{where} \quad \phi_j(t)=\int_{\mathbb{S}^{n-1}}\phi(t,\theta)\chi_j(\theta)\ud\theta.
\end{equation*}
In other terms, $\phi_j$ is the projection of $\phi$ on the eigenspace $V_j$. 
Thus, to understand the kernel of $\mathscr{L}_{\rm cyl}^{\varepsilon}$, we consider the induced family of ODEs
\begin{equation*}
	\mathscr{L}_{\rm cyl}^{\varepsilon,j}(\psi_j)=0 \quad {\rm for} \quad j\in\mathbb{N}.
\end{equation*}

\subsection{Jacobi fields}
Now we investigate the growth/decay rate in which the Jacobi fields on the kernel of the linearized operator grow/decay. 
For this, we transform the sixth order equation \eqref{jacobiequationproj} into a first order equation on $\mathbb{R}^6$. 
More precisely, defining $\Psi=(\psi,\psi^{(1)},\psi^{(2)},\psi^{(3)},\psi^{(4)},\psi^{(5)})$, we conclude that \eqref{jacobiequationproj} is equivalent to the first order system 
\begin{equation*}
	\Psi'_j(t)=N_{\varepsilon,j}(t)\Psi_j(t) \quad {\rm in} \quad \mathbb R^6.
\end{equation*}
Here
\begin{equation*}
	N_{\varepsilon,j}(t)=
	\left[{\begin{array}{cccccc}
			0 & 1 & 0 & 0 & 0 & 0 \\
			0 & 0 & 1 & 0 & 0 & 0\\
			0 & 0 & 0 & 1 & 0 & 0\\
			0 & 0 & 0 & 0 & 1 & 0\\
			0 & 0 & 0 & 0 & 0 & 1\\
			0 & B_j & C_{j} & D_{j}& E_j & F_{\varepsilon,j}(t)\\
	\end{array} } \right],
\end{equation*}
where 
\begin{align*}
	B_j&=-K_4+2\lambda_j\\
	C_j&=\lambda_jJ_3\\
	D_j&=K_2+2J_2\lambda_j\\
	E_j&=\lambda_jJ_1\\
	F_{\varepsilon,j}(t)&=-K_0+J_0\lambda_j-L_0\lambda_j^2+\lambda_j^3-\hat{c}_nv_{\varepsilon}^{\frac{12}{n-6}}(t)
\end{align*}
are the coefficients of the projected Jacobi equation in \eqref{jacobiequationproj}.

Notice that as a consequence of Proposition~\ref{prop:andrade-wei}, one has that $N_{\varepsilon,j}(t)$ is a $T_\varepsilon$-periodic matrix.
Hence, the {monodromy matrix} associated with this ODE system with periodic coefficients is given by $M_{\varepsilon,j}(t)=\exp{\int_{0}^{t}N_{\varepsilon,j}(\tau)\ud\tau}$. Finally, we define the {Floquet exponents}, denoted by $\widetilde{\mathfrak{I}}_j^\varepsilon$,
as the complex frequencies associated to the eigenvectors of $M_{\varepsilon,j}(t)$, which forms a six-dimensional basis for the kernel of $\mathscr{L}_{\rm cyl}^{\varepsilon,j}$. 
Using Abel's identity, we get that $N_{\varepsilon,j}(t)$ is constant, which yields
\begin{equation*}
	\det(M_{\varepsilon,j}(t))=\exp{\int_{0}^{T}\trace{N}_{\varepsilon,j}(\tau)\ud\tau}=\exp\left(-{\int_{0}^{T}C_{\varepsilon,j}(\tau)\ud\tau}\right)=1.
\end{equation*}
Since $N_{\varepsilon,j}(t)$ has real coefficients, all its eigenvalues are pairs of complex conjugates, and so
\begin{equation*}
	\widetilde{\mathfrak{I}}_j^\varepsilon=\{\pm\rho_{1,\varepsilon,j},\pm{\rho}_{2,\varepsilon,j},\pm{\rho}_{3,\varepsilon,j}\}\subset\mathbb C \quad {\rm for} \quad j\in\mathbb N,
\end{equation*}
where 
\begin{equation*}
	\rho_{\ell,\varepsilon,j}=\alpha_{\ell,\varepsilon,j}+i\beta_{\ell,\varepsilon,j} \quad {\rm for} \quad \ell=1,2,3.
\end{equation*}
Then, the set of indicial roots of $\mathscr{L}_{\rm cyl}^{\varepsilon,j}$ are given by 
\begin{equation*}
	\mathfrak{I}_j^\varepsilon=\{\pm\beta_{1,\varepsilon,j},\pm{\beta}_{2,\varepsilon,j},\pm{\beta}_{3,\varepsilon,j}\}\subset\mathbb R \quad {\rm for} \quad j\in\mathbb N.
\end{equation*}
Moreover, for any $\psi\in \ker(\mathscr{L}_{\rm cyl}^{\varepsilon,j})$, we have
\begin{equation*}
	\psi=b^+_1\psi_{\varepsilon,j}^{1,+}+b_1^-\psi_{\varepsilon,j}^{-}+b_2^+\psi_{\varepsilon,j}^{2,+}+b_2^-\psi_{\varepsilon,j}^{2,-}+b_3^+\psi_{\varepsilon,j}^{3,+}+b_3^-\psi_{\varepsilon,j}^{3,-} \quad {\rm for} \quad j\in\mathbb N,
\end{equation*}
where 
\begin{equation*}
	\psi^{1,\pm}_{\varepsilon,j}(t)=e^{\pm\rho_{1,\varepsilon,j} t}, \quad {\psi}_{\varepsilon,j}^{2,\pm}(t)=e^{\pm{\rho}_{2,\varepsilon,j} t} \quad {\rm and} \quad {\psi}_{\varepsilon,j}^{3,\pm}(t)=e^{\pm{\rho}_{3,\varepsilon,j} t} \quad {\rm for} \quad j\in\mathbb N.
\end{equation*}
Hence, the exponential decay/growth rate of the Jacobi fields is controlled by quantity below
\begin{equation*}
	\beta^*_{\varepsilon,j}:=\min\{|\beta_{1,\varepsilon,j}|,|\beta_{2,\varepsilon,j}|,|\beta_{3,\varepsilon,j}|\}.
\end{equation*}
Therefore, the asymptotic properties of $\mathscr{L}_{\rm cyl}^{\varepsilon,j}$ are obtained by the study of $\mathfrak{I}^{\varepsilon}$. 

Let us define indicial root function $\beta:(0,1)\rightarrow (0,+\infty)$ given by $\beta(\varepsilon)=\beta^*_{\varepsilon,0}$
is the lowest indicial root of 
\begin{align*}
	\mathscr{L}_{\rm cyl}^{\varepsilon,0}(\psi)&=\psi^{(6)}-K_4\psi^{(4)}+K_2\psi^{(2)}+\left(-K_0+\hat{c}_nv_{\varepsilon}^{\frac{12}{n-6}}\right)\psi.
\end{align*}
Then, the (fundamental) period $T:(0,1)\rightarrow (0,+\infty)$ is given by
\begin{equation*}
	T(\varepsilon)=\frac{2\pi}{\beta(\varepsilon)}.
\end{equation*}
Let us keep the notation $\beta(\varepsilon)= \beta_\varepsilon$ and $T(\varepsilon)= T_\varepsilon$. 

The next lemma relates the fundamental period of the Delaunay solution with the indicial roots of the linearized operator around this solution.
From this, we show that the period function blows up.
We emphasize that this will be enough for later performing our bifurcation technique.

\begin{lemma}\label{lm:periodbehavvior}       
	The period function $T:(0,1)\rightarrow (0,+\infty)$ is an increasing and satisfies
	\begin{align}\label{monotonicityperiod}
		\lim_{\varepsilon \nearrow 1} T(\varepsilon)=\infty \quad {\rm and} \quad \lim_{\varepsilon \searrow 1} T(\varepsilon)=T_{\rm cyl}.
	\end{align}
	Here $T_{\rm cyl}:=T(\varepsilon_*)$ is the fundamental period of the cylindrical solution given by
	\begin{equation}\label{periodcylinder}
		T_{\rm cyl}=\frac{2\pi}{\beta_{\rm cyl}}
	\end{equation}
	where $\beta_{\rm cyl}\in\mathbb R$
	is the largest indicial root associated with the cylindrical solution. 
	Moreover, one has
	\begin{equation*}
		\sup_{\varepsilon\in(0,1)}\|v_{\varepsilon}\|_{L^{\infty}(\mathbb R)} \leqslant 1.
	\end{equation*} 
\end{lemma}

\begin{proof}
	First, the monotonicity in \eqref{monotonicityperiod} is a direct consequence of Proposition~\ref{prop:andrade-wei} and Lemma~\ref{lm:energyordering}.
	In addition, $\beta_{\rm cyl}=\Im(\rho)\in\mathbb R$ is the greatest imaginary part of complex solutions to the algebraic equation
	\begin{equation*}
		z^6-K_4 z^4+K_2 z^2+\left(-K_0+\hat{c}_n\varepsilon_*^{\frac{12}{n-6}}\right)=0 \quad {\rm in} \quad \mathbb C.
	\end{equation*}
	Using the change of variables $\rho=z^2$, we arrive at 
	\begin{equation*}
		\mathfrak{p}(\rho)=0 \quad {\rm in} \quad \mathbb C,
	\end{equation*}
	where $\mathfrak{p}:\mathbb C\rightarrow\mathbb C$ is the indicial polynomial given by
	\begin{equation*}
		\mathfrak{p}(\rho)=\rho^3-K_4 \rho^2+K_2 \rho+\left(-K_0+\hat{c}_n\varepsilon_*^{\frac{12}{n-6}}\right)=0 \quad {\rm in} \quad \mathbb C.
	\end{equation*}
	Therefore, it is not hard to check that the last equation has one real solution and two complex solutions, which are conjugate.  
	Thus, we find
	\begin{equation}\label{rootcylinder}
		\beta_{\rm cyl}=\beta^*_{\varepsilon_*,0}=\max_{\rho\in \mathfrak{p}^{-1}(\{0\})} \Im(\rho)<+\infty.
	\end{equation}
	The proof is then concluded.
\end{proof}

\subsection{Bifurcation via  Morse index}
In this section, we prove Theorem~\ref{thm1}.
Notice that since $(\mathbb S^1\times\mathbb S^{n-1},g_{\mathbb S^1}\oplus g_{\mathbb S^{n-1}})$ is not Einstein, one cannot directly guarantee that \ref{itm:P} is in force.
However, this can be overcome using the lcf property combined with \eqref{qing-raske}.

\begin{proof}[Proof of Theorem~\ref{thm1}]
	Initially, since $(\mathbb S^1\times\mathbb S^{n-1},g_{\mathbb S^1}\oplus g_{\mathbb S^{n-1}})$ is locally conformally flat and $R(g_{\mathbb S^1}\oplus g_{\mathbb S^{n-1}})\geqslant 0$, it follows from Proposition~\ref{lm:aubin} that \ref{itm:A} holds as well.
	
	Next, for each $T>0$, we consider its associated $T$-periodic Emden--Fowler solution denoted by $v_T:=v_{\varepsilon}\in\mathcal{C}^6(\mathcal C_T)$ solving \eqref{ourPDEcyl} such that $\varepsilon:=v_T(0)\in(0,1)$ is the Delaunay (or necksize) parameter.
	Thus, let us define a conformal metric on $\overline{\mathcal{I}}_T\times\mathbb S^{n-1}$ with $\overline{\mathcal{I}}\subset\mathbb R$ denoting the closed interval $\overline{\mathcal{I}}_T:=[-T/2,T/2]$ of the form $g_{v_\varepsilon}=v_\varepsilon^{{4}/{(n-6)}}g_{\rm cyl}$,       
	where we used the identification 
	\begin{equation}\label{bridgeidentification}
		\mathbb S_{T}\times\mathbb S^{n-1}\simeq \overline{\mathcal{I}}_T\times\mathbb S^{n-1}:=\mathcal{C}^n_{T}.
	\end{equation}
	
	Now notice that since $T\rightarrow+\infty$ as    $\varepsilon\rightarrow0$, one has $v_T\rightarrow v_{\rm cyl}$, where 
	$v_{\rm cyl}$ is the cylindrical solution given by \eqref{cylindricalsolution}.
	Hence, one has that $g_T\rightarrow g_{\infty}$ 
	as $T\rightarrow+\infty$, where         
	$g_{\infty}=g_{\rm cyl}$ is the cylindrical metric, which
	is such that 
	\begin{equation*}
		Q_2({g_{\rm cyl}})\equiv(n-1)(n-2), \quad Q_4({g_{\rm cyl}})\equiv\frac{(n-1)(n^2-2n-2)}{8}, \quad {\rm and} \quad  Q_6({g_{\rm cyl}})\equiv Q_n,
	\end{equation*}
	where we recall that $Q_n>0$ is given by \eqref{const_q_curv}.
	Hence, applying \textcolor{red}{\ref{itm:A}}, we get
	\begin{equation*}
		\mathcal{Y}_6^{+}(\mathbb{S}_{T}^1 \times \mathbb{S}^{n-1},g_{\rm cyl})=\mathcal{Y}_6(\mathbb{S}_{T}^1 \times \mathbb{S}^{n-1},g_{\rm cyl})<\mathcal{Y}_6^{+}( \mathbb{S}^n,g_0).
	\end{equation*}
	Furthermore, we recall that from Lemma~\ref{lm:characterizationofcrititcalpoints}, we know that each critical point of $\mathcal{Q}_6$ in the conformal class $[g_{\rm cyl}]$ must be a constant $Q$-curvature metric on $\mathbb{S}_{T}^1 \times \mathbb{S}^{n-1}$. 
	We pull this constant $Q$-curvature metric on $\mathbb{S}_{T}^1 \times \mathbb{S}^{n-1}$ back to its universal cover $\mathbb{R} \times \mathbb{S}^{n-1}:=\mathcal{C}^n_\infty$, obtaining a smooth, positive $T$-periodic function $v_T: \mathcal{C}_\infty \rightarrow(0, \infty)$ satisfying \eqref{ourODE}. 
	From Lemma~\ref{lm:classification}, it holds that one the three alternatives hold: either the constant cylindrical solution $v_{\rm cyl}$, translates of the spherical solution $v_{\rm sph}$, or translates of a Delaunay solution $v_{\varepsilon}$ for some $\varepsilon \in(\varepsilon_*, 1)$.
	
	For any $T>0$, let us denote by $\mu(T)\in\mathbb N$, the number of $Q$-curvature metrics in the conformal class $[g_{\rm cyl}]$ on $\mathbb{S}_{T}^1 \times \mathbb{S}^{n-1}$.       
	Next, we express this number in terms of the period $T\in(0,1)$.
	
	\noindent{\bf Claim 1:} It holds that $\mu(T):=\lfloor \frac{T}{T_{\rm cyl}} \rfloor:=\kappa$.
	
	\noindent  
	We proceed by induction in the following way. 
	As before, we have the normalization $Q_6(g)=Q_n$.
	The cylindrical solution $v_{\rm cyl}$ is the only solution when $0<T \leqslant T_{\rm cyl}$, where $T_{\rm cyl}$ is given in \eqref{periodcylinder}.
	For $T_{\rm cyl}<T \leqslant 2 T_{\rm c y l}$ we have two constant $Q$-curvature metrics, namely the cylinder and the Delaunay metric with Delaunay parameter $\varepsilon\in(0,1)$ such that $T=T_\varepsilon$.      
	When $2 T_{\rm c y l}<T \leqslant 3 T_{\rm c y l}$ we obtain 3 constant $Q$-curvature metrics, namely the cylindrical solution $v_{\rm c y l}$, the Delaunay solution $v_\varepsilon$ such that $T_\varepsilon=T$, and the Delaunay solution $v_\alpha$ such that $T_\alpha=T / 2$. 
	Continuing inductively, when $(\kappa-1) T_{\rm cyl}<T \leqslant \kappa T_{\rm cyl}$, one can find $\kappa\in \mathbb N$ distinct constant sixth order $Q$-curvature metrics. 
	Namely the cylindrical solution $v_{\rm cyl}$ together with the Delaunay solution $v_{\varepsilon_\ell}$ with $T_{\varepsilon_\ell}=T / \ell$ for each $\ell\in\{1,2, \ldots, \kappa-1\}$, which proves the first claim.
	
	Next, we classify the blow-up behavior of the family above.
	
	\noindent{\bf Claim 2:} It follows 
	\begin{equation}\label{convergenceinftmany}
		\lim_{\varepsilon \nearrow +\infty}\|v_\varepsilon-v_{\rm sph}\|_{\mathcal{C}^\infty_{\rm loc}(\mathbb R)}=0
	\end{equation}
	uniformly on compact subsets, where $v_{\rm sph}$ is the spherical solution given by \eqref{sphericalsolution}.
	
	\noindent In fact, for each $T>T_{\rm cyl}$ the Delaunay solution $v_{\varepsilon}$ such that $T_{\varepsilon}=T$ are such that $v_{\varepsilon}(0)=\varepsilon$ and ${v}^{(1)}_\varepsilon(0)=0$.
	Thus, by Lemma~\ref{lm:energyordering}, these two initial conditions uniquely determine a solution to \eqref{ourODE}.
	Hence, from Lemma~\ref{lm:periodbehavvior}, we know that $\lim_{\varepsilon \nearrow 1} T_\varepsilon=+\infty$, which implies the convergence in \eqref{convergenceinftmany}.     
	Finally, since the estimate $\|v_\varepsilon\|_{\mathcal{C}^\infty(\mathbb R)}<1$ holds for all $\varepsilon>0$,
	this convergence is uniform on compact subsets by the Arzela--Ascoli theorem. 
	This proves the second claim.
	
	In conclusion, we prove the stability statement below.
	
	\noindent{\bf Claim 3:} $v_{\varepsilon_1}$ is the only stable critical point of $\mathcal{Q}_6$ among $\{v_{\rm c y l}, v_{\varepsilon_1}, v_{\varepsilon_2}, \ldots, v_{\varepsilon_{\ell-1}}\}$.
	
	\noindent Indeed, the function $w_{\varepsilon_\ell}=\partial_t {v}_{\varepsilon_\ell}$ satisfies $  \mathscr{L}_{\varepsilon_\ell}(w_{\varepsilon_\ell})=0$ in $\mathbb R$, where $\mathscr{L}_{\varepsilon_\ell}$ is the linearization of \eqref{ourODE} about $v_{\varepsilon_\ell}$. 
	Furthermore, observe that
	\begin{equation*}
		\left\{t \in\overline{\mathcal{I}}_T: w_{\varepsilon_\ell}>0\right\}=\bigcup_{j=-\lfloor \ell/2\rfloor}^{\lfloor \ell/2 \rfloor}\left(j T_{\varepsilon_\ell},\left(\frac{2 j+1}{2}\right) T_{\varepsilon_\ell}\right),
	\end{equation*}
	where $\lfloor \ell/2 \rfloor$ denotes the greatest non-negative integer less than or equal to $\ell/2$. 
	When $\ell \geqslant 2$ the number of nodal domains combined with Strum-Liouville theory implies $-\mathscr{L}_{\varepsilon_\ell}:\mathcal{C}^6(\mathbb R)\rightarrow \mathcal{C}^0(\mathbb R)$ has at least $\ell$ negative eigenvalues, and so $v_{{\varepsilon_\ell}}\in \mathcal{C}^6(\mathbb R)$ cannot be a stable critical point of $\mathcal{Q}_6$. Furthermore the function $w_0\in\mathcal{C}^6(\mathbb R)$ defined as $w_0(t)=\cos (\mu_{\rm cyl} t)$,
	where $\mu_{\rm cyl}>0$ is given by \eqref{rootcylinder}, satisfies $\mathscr{L}_{\rm c y l}(w_0)=0$, where $\mathscr{L}_{\rm c y l}:\mathcal{C}^6(\mathbb R)\rightarrow \mathcal{C}^0(\mathbb R)$ is the linearization of \eqref{ourODE} around $v_{\rm c y l}$. 
	When $T>2 T_{\rm c y l}$ the function $w_0$ has at least 2 disjoint regions on which it is positive, so $v_{\rm c y l}$ cannot be a stable critical point of $\mathcal{Q}_6$ for $T\gg1$ large.
	This proves Claim 2.
	
	From Claims 2 and 3, we conclude that $v_\varepsilon\in \mathcal{C}^\infty(\mathbb{S}_{T_\varepsilon}^1 \times \mathbb{S}^{n-1})$ minimizes $\mathcal{Q}_6$ over the conformal class $[g_{\rm cyl}]$, which yields
	\begin{align*}
		\mathcal{Y}_6^{+}(\mathbb{S}_{T_\varepsilon}^1 \times \mathbb{S}^{n-1},g_{\rm cyl}) & =\mathcal{Q}_6(g_{v_\varepsilon})= \frac{\frac{2}{n-6}\int_{\mathbb{S}^1 \times \mathbb{S}^{n-1}} Q_6({g_{v_\varepsilon}}) \ud \mathrm{v}_{g_{v_\varepsilon}}}{\operatorname{vol}_{g_{v_\varepsilon}}(\mathbb{S}^1 \times \mathbb{S}^{n-1})^{\frac{n-6}{n}}} \\
		& =\frac{2Q_n}{n-6} \frac{\operatorname{vol}_{g_{v_\varepsilon}}(\mathbb{S}^1 \times \mathbb{S}^{n-1})}{\operatorname{vol}_{g_{v_\varepsilon}}\left(\mathbb{S}^1 \times \mathbb{S}^{n-1}\right)^{\frac{n-6}{n}}} \\
		& =\frac{2 Q_n }{n-6}\frac{\int_{\mathcal{I}_{T_\varepsilon}} \int_{\mathbb{S}^{n-1}} v^{\frac{2 n}{n-6}}\ud \theta \ud t}{\left(\int_{\mathcal{I}_{T_\varepsilon}} \int_{\mathbb{S}^{n-1}} v^{\frac{2 n}{n-6}} \ud \theta \ud t\right)^{\frac{n-6}{n}}} \\
		& =\frac{2Q_n\omega_{n-1}^{n / 6}}{n-6} \left(\int_{\mathcal{I}_{T_\varepsilon}} v^{\frac{2 n}{n-6}} \ud t\right)^{n/6},
	\end{align*}
	where $\omega_{n-1}=|\mathbb{S}^{n-1}|$ denotes the $n$-dimensional Lebesgue measure of the standard Euclidean sphere.
	Finally, we let $\varepsilon \nearrow 1$ in the last identity to see
	\begin{align*}
		\mathbb{Y}_6^{+}\left(\mathbb{S}^1 \times \mathbb{S}^{n-1}\right) & \geqslant \lim_{\varepsilon\nearrow 1} \mathcal{Y}_6^{+}(\mathbb{S}_{T_\varepsilon}^1 \times \mathbb{S}^{n-1},g_{\rm cyl}) \\
		& =\frac{2Q_n\omega_{n-1}^{n / 6}}{n-6}\left(\int_{-\infty}^{+\infty}v_{\rm s p h}(t)^{\frac{2 n}{n-6}} \ud t\right)^{n / 6} \\\label{jessehint2}
		& =\frac{2Q_n\omega_{n-1}^{n / 6}}{n-6}\left(\int_{-\infty}^{+\infty}\cosh (t)^{-n} \ud t\right)^{n / 6} \\
		& =\frac{2Q_n\omega_{n-1}^{n / 6}}{n-6}\left(\int_0^{+\infty}\left(\frac{1+r^2}{2}\right)^{-n} r^{n-1} \ud r\right)^{n / 6}\\
		&=\mathcal{Q}_6(g_0)\\
		&=\mathcal{Y}_6^{+}(\mathbb{S}^n,g_0)\\
		&=\mathbb{Y}_6^{+}(\mathbb{S}^n),
	\end{align*}
	where $r=e^{-t}$. 
	This, in turn, completes our first main result.
\end{proof}

\section{The noncompact case}\label{sec:noncompactcase}
This section is based on a topological argument from \cite{MR3504948} to prove Theorem~\ref{thm2} as its corollaries.



\subsection{Riemannian coverings}
Notice that in this case the variational formulations as in Lemma~\ref{lm:characterizationofcrititcalpoints} are no longer available. 
To circumvent this issue and establish nonuniqueness results in this context, we combine infinite towers of coverings with the existence of a positive Green function \ref{itm:P} and the reversed Aubin inequality \ref{itm:A'}.

\begin{proof}[Proof of Theorem~\ref{thm2}]
	From the assumption on the group of deck transformations, we can find an infinite tower of Riemannian coverings satisfying
	\begin{equation*}
		(\widetilde{M}^n_{\infty}, \tilde{g}_{\infty}) \longrightarrow \cdots \longrightarrow(\widetilde{M}^n_\ell, \tilde{g}_\ell) \longrightarrow \cdots \longrightarrow(\widetilde{M}^n_2, \tilde{g}_2) \longrightarrow(\widetilde{M}^n_1, \tilde{g}_1) \longrightarrow(\widetilde{M}^n_0, \tilde{g}_0),
	\end{equation*}
	where $\widetilde{M}^n_\ell \rightarrow \widetilde{M}^n_0$ is an $d_\ell$-sheeted covering and $(\widetilde{M}^n_0, \tilde{g}_0)=(M^n, g)$, with $d_\ell \geqslant 2$ and $d_\ell \nearrow +\infty$ as $\ell\rightarrow +\infty$.
	Since $(\widetilde{M}^n_\ell, g_\ell)$ are locally isometric to $(\widetilde{M}^n_0, g_0)$. 
	Using the same argument as in \cite[Lemma 3.6]{MR2301449}, it is not hard to check that $(\widetilde{M}^n_\ell, \tilde{g}_\ell)$ also satisfies \ref{itm:P} and \ref{itm:A'}, which gives us
	\begin{equation*}
		{\Theta}_6(\widetilde{M}_\ell, \tilde{g}_\ell)=\frac{\int_{\widetilde{M}_\ell} Q_6(\tilde{g}_\ell) \ud\mathrm{v}_g}{\|Q_6(\tilde{g}_\ell)\|_{L^{\frac{2 n}{n+6}}(\widetilde{M}_\ell)}}
	\end{equation*}
	attains its maximum in each conformal class $[\tilde{g}_\ell]$ at some metric $\bar{g}_\ell\in[\tilde{g}_\ell]$ for all $\ell\in\mathbb N$. 
	Denote by $\tilde{h}_\ell$ the pullback to $\widetilde{M}_\ell$ of the metric $\bar{g}_1$, and note that $\tilde{h}_\ell \in[\tilde{g}_\ell]$. 
	
	Using that $Q_1=Q_6({\tilde{g}_1})>0$ is a positive constant and $\tilde{h}_\ell$ is locally isometric to $\tilde{g}_1$, we get that $Q_6({\tilde{g}_\ell})=Q_1$ for all $\ell \geqslant 1$, which yields
	\begin{align*}
		{\Theta}_6(\widetilde{M}_\ell, \tilde{g}_\ell) =\frac{\int_{\widetilde{M}_\ell} Q_6({\tilde{h}_\ell}) \ud\mathrm{v}_{\tilde{h}_\ell}}{\|Q_6({\tilde{h}_\ell})\|^2_{L^{\frac{2 n}{n+6}}(M)}}=Q_1^{-1} {\rm vol}_{\tilde{g}_1}(\widetilde{M}_1)^{-6/n}\left(\frac{d_\ell}{d_1}\right)^{-6/n} \searrow 0 \quad  {\rm as} \quad \ell \rightarrow +\infty.
	\end{align*}      
	From this, one can also see 
	\begin{equation*}
		\frac{2}{n-4} {\Theta}_6(\widetilde{M}_\ell, \tilde{h}_\ell)<\Theta_6(\mathbb{S}^n, g_0),
	\end{equation*}
	if $\ell \geqslant \ell_0$ for some $\ell_0 \gg 1$ sufficiently large.
	Hence, $h_{\ell_0}$ is not a maximizer of $\Theta_6$ in $[\tilde{g}_{\ell_0}]$. 
	By \ref{itm:A'}, there exists a maximizer of $\Theta_6$ in $[\tilde{h}_{\ell_0}]$, which is not homothetic to $[\tilde{h}_{\ell_0}]$. 
	The pullbacks to $\widetilde{M}^n_{\infty}$ of this maximizer and $\tilde{h}_{\ell_0}$ are two nonhomothetic complete metrics with constant positive $Q$-curvature in $[\tilde{g}_{\infty}]$. Now, we can do the same procedure replacing $(\widetilde{M}^n_1, \tilde{g}_1)$ with $(\widetilde{M}^n_{\ell_0}, \tilde{g}_{\ell_0})$.
	Therefore, by interactively repeating this process, the theorem is proved.
\end{proof}

\begin{remark}
	For example, as a consequence of Selberg--Malcev Lemma \cite[Section~7.5]{MR1299730}, any hyperbolic manifold has an infinite profinite completion due to the hence it admits finite-sheeted Riemannian coverings by hyperbolic manifolds of arbitrarily large volume.
	This fact also holds for symmetric spaces of hyperbolic or Euclidean type in general. 
	In this regard, the second summand $(\mathcal{S}^{n_2},g_{\mathcal{S}})$ represents either the standard Euclidean space $(\mathbb R^n,g_{\mathbb R^n})$ or Hyperbolic space $(\mathbb H^n, g_{\mathbb H^n})$.
\end{remark}

\subsection{Bifurcation via infinite tower of Riemannian coverings}   
Here we apply the last proposition to prove the second main result of the manuscript:

\begin{proof}[Proof of Corollary~\ref{cor3}]
	Initially, using a classical result from \cite{MR146301}, it follows that every symmetric space $\mathcal S^{n_2}$ of Euclidean type admits irreducible compact quotients $\Sigma=\mathcal{S}^{n_2}/ \Gamma$, and the same is valid for the Euclidean space. 
	Fixing, such a compact quotient $(\Sigma^{n_2}, g_{\Sigma})$ with the induced locally symmetric metric. 
	Since $\pi_1(\Sigma)$ is infinite and residually finite, it has infinite profinite completion. 
	Hence, it follows that $(M^n_{\infty}, g_{\infty})=(C^{n_1} \times \mathcal{S}^{n_2}, g_C \oplus g_{\mathcal{S}})$, and $(M^n_0, g_0)=(C^{n_1} \times \Sigma^{n_2}, g_C \oplus g_\Sigma$).
	Under the assumptions \ref{itm:P} and \ref{itm:A'}, the proof is finished as a direct application of Theorem~\ref{thm2}.
\end{proof}

\begin{proof}[Proof of Corollary~\ref{cor4}]
	First, we recall the conformal equivalence below
	\begin{equation*}
		\mathbb{S}^n \setminus \mathbb{S}^k \simeq \mathbb{S}^{n-k-1}\times \mathbb{H}^{k+1} \simeq \mathbb{S}^{n_1} \times \mathbb{H}^{n_2} = \widetilde{(\mathbb{S}^{n_1} \times \mathbb{M}^{n_2}(T))}:=\widetilde{M}^n_T
	\end{equation*}
	where $n_1=n-k+1$, $n_2=k+1$ and $\mathbb{M}^{n_2}(T)=\mathbb{H}^{n_2}\backslash \Gamma(T)$ being the quotient manifold under the periodic Kleinian group $\Gamma(T)\subset \mathbb{H}^{n_2+1}$ arising as holonomy representation of the fundamental group $\pi_1(M)$ with $(M^n,g)=(\mathbb{S}^{n_1} \times \mathbb{H}^{n_2},g_1\oplus g_2)$. 
	Second, using the locally conformally flat condition together with $R(g)\geqslant0$ and \eqref{qing-raske}, one can apply Proposition~\ref{prop:greenfunction1} to conclude that \ref{itm:A} is also true, which combined with \ref{itm:P} is enough to prove \ref{itm:A'}.
	Finally, we finish the proof by direct applying Corollary~\ref{cor3}. 
	In this case, the family $\{\widetilde{M}^n_T\}_{T>0}$ is a tower of finite-sheeted regular coverings with arbitrarily large volume.		
\end{proof}
\color{black}

\begin{remark}\label{rmk:connection}
	Finally, let us come back to the case $k=0$, for which one has the following conformal equivalence $\mathbb{S}^n \setminus \mathbb{S}^0 \simeq \mathbb{R}^n \setminus\{0\} \simeq \mathbb{S}^{n-1} \times \mathbb{R}$ with the identification $\mathbb{S}^0:=\{p,-p\}$ with $p=\mathbf{e}_1\in\mathbb S^n$ the north pole.
	Notice that the proof above can be adapted {\it mutatis mutandis}, using the compact quotients $M^n_T=\mathbb{S}^1(T)\times \mathbb{S}^{n-1} $ as in \eqref{bridgeidentification}, where $\mathbb{S}^1(T)=\mathbb{R} / T\mathbb{Z}$, which as above form a similar tower of finite-sheeted regular coverings with arbitrarily large volume. 
	Therefore, our technique produces minimizing solutions at each level that lift to pairwise nonhomothetic periodic solutions in the universal covering $\widetilde{M}^n_T=\mathbb{S}^{n-1} \times \mathbb{R}=\mathcal{C}^n_\infty$.
	This, in turn, shows that Corollary~\ref{cor3} is indeed an extension of Theorem~\ref{thm1} to a much broader situation.
\end{remark}

\begin{remark}
	Using this covering approach and
	motivated by the seminal work \cite{MR931204}, it would be interesting to understand under which conditions on the manifold $(M^n,g)$ the equality holds $\mathcal{M}_6(g_C\oplus g_{\mathcal{S}})=\mathcal{M}_{6,\Lambda}(g)$,
	where $\Lambda^k\subset M^n$ is $k$-dimensional submanifold and we recall that $g_k=g_{\mathbb S^n\setminus \mathbb S^k}$ is the standard round metric singular along an equatorial subsphere.
	For instance, when $M$ descends to a quotient $M/\Gamma$ for some finite subgroup $\Gamma\subset {\rm iso}(M)$.
	The most simple example is the locally conformally flat manifold $(M^n,g)\simeq(\mathbb S^n,g_0)$ and $\Lambda:=\mathbb S^n\setminus\mathscr{D}(\widetilde{M})$ is a Kleinian group arising as the image of the developing map $\mathscr{D}:\widetilde{M}\hookrightarrow\mathbb S^n$, which by definition is conformal and injective.
	This class of manifolds recovers the example given in Remark~\ref{rmk:connection}.
\end{remark}


\begin{acknowledgement}
	This paper was finished when the first-named author held a Post-doctoral position at the University of British Columbia, whose hospitality he would like to acknowledge.
	This project started when the first and second named authors visited the College University of New York, which they are thankful to R. G. Bettiol for the invitation.
	We would like to thank J. Ratzkin for reading the final manuscript and his compelling contributions.
	The authors warmly thank J. S. Case for his several remarks after the first version of this preprint was on Arxiv.
\end{acknowledgement}


\end{document}